\documentclass[a4paper,10pt,leqno]{article}
\usepackage{amsfonts,amssymb,amsmath,amsthm}
\usepackage{fullpage}
\usepackage{epic,eepic}

\newtheorem{theorem}{Theorem}[section]
\newtheorem{proposition}[theorem]{Proposition}
\newtheorem{corollary}[theorem]{Corollary}
\newtheorem{lemma}[theorem]{Lemma}
\newtheorem{example}[theorem]{Example}
\theoremstyle{remark}
\newtheorem{remark}[theorem]{{\bf Remark}}

\numberwithin{equation}{section}

\newcommand{\A}{\mathbf A}
\newcommand{\N}{{\mathbb N}}
\newcommand{\R}{{\mathbb R}}

\renewcommand{\P}{{\mathcal P}}

\newcommand{\Ge}{{\mathcal G}}

\newcommand{\eps}{{\varepsilon}}
\newcommand{\vark}{{\varkappa}}

\newcommand{\ol}{\overline}

\renewcommand\a{\alpha}
\renewcommand\b{\beta}
\newcommand\g{\gamma}

\renewcommand\d{\delta}
\renewcommand\l{\lambda}
\newcommand\n{\nabla}
\renewcommand\o{\omega}\renewcommand\O{\Omega}

\def\var{\varphi}

\def\vark{\varkappa}

\newcommand{\esssup}{\mathop{\mathrm{ess\,sup}}}

\renewcommand{\div}{\operatorname{div}}

\newcommand{\indi}{\mbox{\rm 1\hspace{-2pt}\rule[0mm]{0.2mm}{7.9pt}\hspace{ 1pt}}}
\newcommand{\ind}{\hbox{\rm 1\hskip -4.5pt 1}}

\def\part{\partial}

\def\leq {\leqslant}
\def\geq {\geqslant}
\def\pr{^{\prime}}
\newcommand{\tr}{\operatorname{tr}}
\newcommand{\TR}[1]{\tr\{{#1}\}}

\def\lan{\langle}
\def\ran{\rangle}
\newcommand{\<}{\langle}
\renewcommand{\>}{\rangle}

\def\XXint#1#2#3{{\setbox0=\hbox{$#1{#2#3}{\int}$}
\vcenter{\hbox{$#2#3$}}\kern-.5\wd0}}

\def\W{\rlap{$\buildrel \circ \over W$}\phantom{W}}

\begingroup\makeatletter\ifx\SetFigFont\undefined%
\gdef\SetFigFont#1#2#3#4#5{%
  \reset@font\fontsize{#1}{#2pt}%
  \fontfamily{#3}\fontseries{#4}\fontshape{#5}%
  \selectfont}%
\fi\endgroup%

\begin{document}

\title{
Gradient estimates for degenerate quasi-linear parabolic
equations }
\author{
{\large Vitali Liskevich}\\
\small Department of Mathematics\\
\small University of Swansea\\
\small Swansea SA2 8PP, UK\\
{\tt v.a.liskevich@swansea.ac.uk}\\
\and
{\large Igor I.\,Skrypnik}\\
\small Institute of Applied\\
\small Mathematics and Mechanics\\
\small Donetsk 83114,  Ukraine\\
{\tt iskrypnik@iamm.donbass.com}
\and
{\large Zeev Sobol}\\
\small Department of Mathematics\\
\small University of Swansea\\
\small Swansea SA2 8PP, UK\\
{\tt z.sobol@swansea.ac.uk}
}

\date{}

\maketitle

\begin{abstract}
For a general class of divergence type quasi-linear degenerate
parabolic  equations with differentiable structure and lower order
coefficients form small with respect to the
Laplacian we obtain $L^q$-estimates for the gradients of solutions,
and for the lower order coefficients from  Kato-type classes we show
that the solutions are Lipschitz continuous with respect to the
space variable.
\end{abstract}

\bigskip

\section{Introduction and main results}

In this paper we study regularity of local weak solutions to general divergence type quasi-linear degenerate parabolic  equations with
measurable coefficients and lower order
terms. This class of equations has numerous applications and has been attracting attention for several decades (see, e.g.  the monographs
\cite{DiB, LaSU, WUZ}, survey \cite{DiUV} and references therein).

Let $\O$ be a domain in $\R^N$, $T>0$. Set $\O_T=\O\times (0,T)$. We study
solutions to
the
equation
\begin{equation}
\label{e0} u_t-\div \A (x,t,u, \n u)=b(x,t,u,\n u) ,\quad (x,t)\in
\O_T.
\end{equation}

Throughout the paper we suppose that 
the function $(\A,b):\O_T\times\R\times \R^N\to \R^N\times \R$
satisfy the Carath\'{e}odory condition, that is $(\A,b)(\cdot,u,z)$
is Lebesgue measurable for all $u\in \R,z\in\R^N$,
and $(\A,b)(x,t,\cdot,\cdot)$ is continuous for almost all $(x,t)\in\O_T$.

We also assume that the following structure conditions are satisfied:
\begin{eqnarray}
\nonumber
\A (x,t,u, z)z &\ge& c_0 |z|^p,\quad z\in\R^n,\\
|\A (x,t,u, z)|&\le& c_1(|z|^{p-1}+1),\label{e1.2b}\\
\nonumber |b(x,t,u,z)|&\le& g(x)|z|^{p-1}+f(x)(|u|^{p-1}+1),
\end{eqnarray}
where $p  \ge 2$,
$c_1$ and $c_2$ are positive constants and $f$ and $g$ are
nonnegative functions.

Let us remind the reader of the notion of a weak solution to
equation \eqref{e0}. We say that $u$ is a weak solution to
\eqref{e0} if $u\in V(\O_T):=L^p_{loc}\big((0,T);
W^{1,p}_{loc}(\O)\big)\cap C\big((0,T); L^2_{loc}(\O)\big)$ and
for any interval $[t_1,t_2]\subset (0,T)$ the integral identity
\begin{equation}
\label{e1.12b}
\int_{\O}u\psi dx\Big|_{t_1}^{t_2}+\int_{t_1}^{t_2}
\int_\O \left\{-u\partial_\tau\psi+\A(x,t,u,\n u) \n\psi
\right\}dx\,d\tau=\int_{t_1}^{t_2}
\int_\O b(x,t,u,\n u)\psi dx\,d\tau
\end{equation}
for any $\psi \in W^{1,p}_c(\O_T)$.

In \cite{LS2} local boundedness of weak solutions to \eqref{e0}
was obtained under optimal conditions on $f$ and $g$ in terms
of membership to the nonlinear Kato classes, which are defined
below. The main thrust of the result in \cite{LS2}  is the
presence of singular lower order coefficients in the structure conditions with optimal assumptions while
not assuming anything in addition on the diffusion part.

In what follows we use the notion of the Wolff potential of a function $f$ (cf.~\cite{AH}), which is defined
by
\[
W_{\b,p}^f(x,R):=\int_0^R \frac{dr}{r}\left(\frac1{r^{N-\b p}}\int_{B_r(x)}f(y)dy\right)^\frac1{p-1},
\]
where here and below $B_r(x)=\{z\in \O\,:\,|z-x|<r\}$, \ $p>1$ and $\b>0$. For the case $\b=1$ it is customary to drop the first index and write $W_{p}^f(x,R)$.
The corresponding non-linear (local) Kato-type  classes $K_{\b,p}$ are defined by
\begin{equation}
\label{e1.9}
K_{\b,p}:=\left\{f\in L_{loc}^1(\O)\,:\,
\lim\limits_{R\to 0} \sup\limits_{x\in\O'}W_{\b,p}^f(x,R)=0 \mbox{ for all }\O'\Subset\O
\right\}.
\end{equation}
In case $\b=1$ we simply write $K_p=K_{1,p}$.
The nonlinear Kato class $K_p$ was introduced in \cite{biroli0}.
As one can easily see,
for $p=2$, the class $K_p$ reduces to the standard definition of the Kato  class with respect to the Laplacian \cite{Bsimon}, which is extensively used
in the qualitative linear theory of elliptic and parabolic second order PDEs. The class $K_p$ turns out to be almost optimal condition on the lower order
coefficients also in case of nonlinear $p$-Laplacian type elliptic and parabolic PDEs for a number of qualitative properties to hold (see \cite{LS1,LS2} and the references therein).
A typical example of a singular function in $K_p$ is
$\displaystyle
\frac{\ind_{B_{1/2}(0)}}{|x|^p\left(\log\frac1{|x|}\right)^\a}$ with
$\a>p-1$, where
here and further on $\ind_{S}$ stands for the characteristic function of the set $S$.
It was proved in \cite{LS2} that the condition $f,\,g^p\in K_p $ implies that $u\in L^\infty_{loc}$. In fact, an inspection of the proof there
shows that the conditions of membership of the structure coefficients to the corresponding
Kato class can be weakened to the requirement that $\sup\limits_{x\in {\O'}}W_p^{g^p+f}(x,2R)$ is
sufficiently small for any subdomain $\O'\Subset \O$. More precisely, (cf.~\cite{LS2},~\cite{LSS})
\begin{quote}
  there exists $\nu>0$ such that, if for every subdomain $\O'\Subset \O$
\begin{eqnarray}
 \label{e1.7b}
 \lim_{R\to 0}\sup\limits_{x\in {\O'}}W_p^{g^p+f}(x,R) <\nu,
\end{eqnarray}
then $u\in L^\infty_{loc}(\O_T)$.
\end{quote}
Throughout the paper we assume that $f$ and $g$ satisfy a condition
guaranteeing that  $u\in L^\infty_{loc}(\O_T)$.

We would like to remark that the condition of smallness of $\sup\limits_{x\in {\O'}}W_p^{f}(x,2R)$ cannot be distinguished
from the Kato type condition $\lim\limits_{R\to 0}\sup\limits_{x\in {\O'}}W_p^{f}(x,R) =0$ if $f$ has only isolated singularities.
For $p=2$ this was already noticed in \cite{AS}, where the corresponding example was constructed.
We give an extension of this example for the general $p\in [2,N)$ in the Appendix.

In this paper
we are interested in the estimates of the gradients of
solutions to \eqref{e0} with differentiable structure in the
diffusion part. The problem of higher regularity of solutions
of quasi-linear equations (and systems) has a long history,
which started from $C_{loc}^{1,\a}$ results for homogeneous
elliptic equations (we refer the reader to the well known
monographs \cite{giusti,La,LaSU,MZ} for the basic results,
historical surveys and references). For a general structure
divergence type quasi-linear elliptic equations, the H\"older
continuity of the gradients of solutions were obtained by
DiBenedetto~\cite{DiB83} and Tolksdorf~\cite{tolksdorf}. For
the case of quasi-linear parabolic equations gradient estimates
under different conditions were studied in
\cite{DiF85,Lieb,Lieb1}, see also monographs \cite{DiB,WUZ} for
basic results and some historic comments. Very recently several
interesting results on estimates of the gradients of solutions
to quasi-linear elliptic and parabolic equations via nonlinear
potentials were obtained in \cite{DM2,DM3}. Most of the
results in \cite{DM2,DM3} concern the elliptic equations of
$p$-Laplacian type $-\div \A (x, \n u)=\nu$ with a measure in
the right hand side. The authors give pointwise estimates of
the gradients of solutions via a nonlinear Wolff potential of
the measure $\nu$, and as a consequence obtain a sufficient
condition for the boundedness of the gradient. In \cite{DM2}
also parabolic equations were studied, and pointwise estimates
of solutions and gradients were obtained, but only for the case
$p=2$. While the results in \cite{DM2,DM3} nicely cover the
case of general measures on the right hand side, the situation
becomes different when the measure $\nu$ is absolutely
continuous with respect to the Lebesgue measure with locally
square integrable density, i.e. $\nu=fdx$ with $f\in
L^2_{loc}$, and the condition on $f$ in \cite{DM2} turns out to
be not optimal, which can be seen on explicit examples. We
remark that while this paper was already in preparation, the
authors were informed about the new preprint \cite{DMl}, where
this situation was studied for the elliptic equations and
systems, and with the vector field $\A$ depending on $\n u$
only. The estimates obtained there are expressed in terms of a
new potential which is in fact $W^{f^2}_{\frac23,3}(x,R)$ and which will
appear in our main results as well.
Below we make a further
comparison of our results with \cite{DMl}.
%
%
%
%
%
%
%
%

We study a general situation for equation \eqref{e0}, that is we allow for the vector field
$\A$ in the diffusion part as well as for the right hand side $b$ to depend on all the arguments. To study higher differentiability
it is standard to assume
%
that
$\A$ is differentiable in $x,u$ and $z$ and that the following ellipticity and growth conditions hold:
\begin{eqnarray}
\label{e1}
&&\lan (\part_z\A) \mu,\mu\ran \ge c_0 |z|^{p-2}|\mu|^2, \quad \forall \mu,\,z\in\R^N,\\
\label{e2}
&& |\part_z\A|\le c_1 (|z|^{p-2}+1),\\
\label{e3}
&&|\part_u\A| \le g_1(x)|z|^{p-2}+f_1(x),\\
\label{e3a}
&&|\part_x\A|\le g_2(x)|z|^{p-1}+ f_2(x),
\end{eqnarray}
where $f,f_1,f_2,g,g_1,g_2$ are nonnegative functions.
Without loss of generality, we do not assume dependence of $u$
in the right hand side of \eqref{e1} - \eqref{e3a} since $u$ is
locally bounded due to \eqref{e1.7b}. In the sequel we refer to
$f,f_1,f_2,g,g_1,g_2$ as to the structure coefficients (cf., e.g.
\cite[Chap.\,~VIII]{DiB}, see also Remark~\ref{rem-dz} below).

Our aim here
is to reveal most general conditions on the structure coefficients guaranteeing
higher integrability and boundedness of the gradients of solutions.
To formulate our results,
we need to introduce 
some additional classes playing special roles in the results.

The class $K_{\frac23,3}$ defined in \eqref{e1.9} with
$\b=\frac23$ and $p=3$, with a typical example of a singular
function in $K_{\frac23,3}$ as $\displaystyle
\frac{\ind_{B_{1/2}(0)}}{|x|^2\left(\log\frac1{|x|}\right)^\a}$
\ with $\a>2$, already appeared in structure conditions in
\cite{LS1} as $\widetilde K_2$.

%

We also need to introduce a class of form bounded function with respect to the Laplacian with form bound $\b>0$, which we further denote by $PK_\b$.

We say that $F$ is form bounded with respect to the Laplacian
with form bound $\b>0$ and write $F\in PK_\b$ if $F\in
L^1_{loc}(\O)$ and there exists $C\ge 0$ such that
for all $\theta\in C_0^\infty(\O)$
\[
\left|\int_\O F \theta^2dx  \right|\le \b \int_{\O}|\n \theta|^2dx +C\int_\O \theta^2 dx.
\]
We will also need the class of infinitesimally form bounded
function with respect to the Laplacian, which we further denote
by $PK_0$, and the class of form bounded function with respect
to the Laplacian, which is denoted by $PK$. These classes are
defined by $PK_0=\bigcap\limits_{\b>0}PK_\b$ and
$PK=\bigcup\limits_{\b>0}PK_\b$. All the three classes became
indispensable in many problems in PDE theory. Their complete
characterization can be found in \cite{MaVe},~\cite{MaVe1}. For comparison
with the Kato type classes, an example of a singular function
in $PK$ is $\displaystyle\frac{\ind_{B_{1/2}(0)}}{|x|^2}$,
while and example of a member of $PK_0$ is $\displaystyle
\frac{\ind_{B_{1/2}(0)}}{|x|^2\left(\log\frac1{|x|}\right)^\a}$
\ with $\a>0$. We also need local versions of the above
classes. Namely, we say that $F\in PK_\b^{loc}$ (respectively,
$PK_0^{loc}$, $PK^{loc}$) if $F\ind_{\O'}\in PK_\b$
(respectively, $PK_0$, $PK$) for any
$\O'\Subset\O$.




While our main object in this paper is the general equation, it seems worth
giving an example of a simpler equation which would illustrate the results below, and which seems to be of independent interest.
Let us consider the nonhomogeneous evolution $p$-Laplace equation $u_t-\Delta_p u =f$. It follows from our results below that
if $f^2\in PK_0$ then the gradient of any weak solution $u$ is in $L^q_{loc}$ for any $q<\infty$, while if $f^2\in K_{\frac23,3}$ then $\n u\in L^\infty_{loc}$.
So, for $f(x)=\frac1{|x|\left(\log\frac1{|x|}\right)^\a}\indi_{B_{1/2}(0)}$ with $\a>0$ we have that $\n u\in L^q_{loc}$ for any $q<\infty$, and with $\a>1$,
the conclusion is that $\n u \in L^\infty_{loc}$, and hence every solution is locally Lipschitz continuous with respect to the spatial variables.

\bigskip
Our strategy is the following. We first show that under some
general assumptions on the structure coefficients there exists
a local weak solution to \eqref{e0} whose space Hessian exists almost everywhere and the
space gradient is
in $L^q_{loc}$ for an arbitrary large $q$, in a cylinder $Q=B_R\times (t_1,t_2) \Subset \O_T$ provided
the Wolff potentials $\sup_{x\in B_R}W^f_p(x,2R)$ and $\sup_{x\in B_R}W^{g^p}_p(x,2R)$ are sufficiently small. This constitutes an
existence result. The required a priori estimates are obtained
by a finite number of iterations of Moser type. The main
assumption here is that all squares of structure coefficients
are infinitesimally form bounded with respect to the Laplacian.
Next, under some mild additional assumption on $f_1$ and
$g_1$, for instance, $f_1,g_1^p\in K_p$, we prove that every weak solution to \eqref{e0} in $\O_T$ has the
same smoothness. In the proof of this result we follow the idea
of Tolksdorf \cite{tolksdorf}, comparing the solution $u$ to
\eqref{e0} on a small cylinder, with a smooth solution to an
auxiliary initial boundary value problem in $Q$ with $u$ as
initial boundary value data and the equation satisfying the
same structure condition as \eqref{e0}. A significant
difference between our situation and that in \cite{tolksdorf}
is that we do not rely on a priori H\"older continuity (or even
continuity) of the weak solution to \eqref{e0} but rather on
the property of smallness of the Wolff potentials $\sup_{x\in B_R}W^{g^p}_p(x,2R)$ and $\sup_{x\in B_R}W^{g_1^p}_p(x,2R)$
(see Lemma~\ref{lem-Dir}). The next
step is to obtain the supremum estimates of the gradient. This
requires stronger assumptions on the structure coefficients.
The technique we use to achieve the result is a parabolic
version of the Kilpel\"ainen--Mal\'y technique \cite{KiMa},
\cite{MZ} (see~\cite{LS2,Skr1}).

Our first result concerns the existence of weak solutions to
\eqref{e0} with integrable powers of the gradient. Further on
we distinguish between the gradient $\nabla\xi$ of a scalar
function $\xi$ and the spatial derivative $D\zeta$ of a vector
valued function $\zeta$. We set
$[D\zeta]_{kl}=\partial_{x^l}\zeta_k$. The space $\R^{N\times
N}$ of matrices is equipped with the Hilbert-Schmidt norm: for
$M=\{m_{kl}\}\in\R^{N\times N}$, we set
$|M|^2\equiv|M|^2_{HS}=\sum\limits_{kl} m_{kl}^2.$

\begin{theorem}
\label{PK-existence} Let $Q$ denote the cylinder $B_R\times
(t_1,t_2)$ such that $Q\Subset \O_T$. Let $v\in V(\O_T)\cap
L^{p'}_{loc}\big((0,T);\;W^{-1,p'}_{loc}(\O)\big)$, and let
$\A$ and $b$ satisfy the structure conditions \eqref{e1.2b} and
\eqref{e1}--\eqref{e3a} with
$(f^2+f_1^2+g_1^2+f_2^2+g_2^2)\ind_{B_R}\in PK$.
Assume that $\sup\limits_{x\in {B_R}}W_p^{g^p+f}(x,2R)$ is
sufficiently small. Then there exists a solution $u$ to
\eqref{e0} in $Q$ satisfying $u=v$ on the parabolic boundary
$\P Q$ of $Q$, such that, for every $l>0$ and $q \ge p$ and
every cylinder $Q'=B'\times (t_1',t_2')\Subset Q$ there exist
constants $\beta,\g$ such that
\begin{equation}
\label{main-est}
\esssup_{t\in(t'_1,t'_2)}\int_{B'} |\n u|^{q-p+2}dx
+\iint_{Q'} \Big|D\big(\n u(|\n u|-l)_+^{\frac q2 -1}\big)\Big|^2)dxd\tau\le \g,
\end{equation}
provided $(g_1^2+g_2^2)\ind_{B_R}\in PK_\beta$. In particular,
if $(g_1^2+g_2^2)\ind_{B_R}\in PK_0$ then $\n u\in
L^q_{loc}(Q)$ for every  $q<\infty$.

Moreover, there exist sequences of Carath\'{e}odory functions
$(\A_n, b_n):Q\times\R\times\R^N\to\R^N\times\R$ and $u_n\in
L^2_{loc}\big((t_1,t_2);\;W^{2,2}_{loc}(B_R)\big)\cap
C\big((t_1,t_2);\;W^{1,2}_{loc}(B_R)\big)\big)$ satisfying
$\partial_tu_n - \div A_n(u_n,\n u_n)=b_n(u_n,\n u_n)$, $u_n=v$
on $\P Q$, such that $(\A_n, b_n)(x,t,s,z)\to (\A, b)(x,t,s,z)$
as $n\to\infty$ for a.a. $(x,t)\in Q$ and all
$(s,z)\in\R\times\R^N$, that $u_n\to u$ and $\n u_n\to \n u$ as
$n\to\infty$ pointwise a.e. on $Q$, and that $(\A_n, b_n)$
satisfies the structure conditions \eqref{e1.2b} and
\eqref{e1}--\eqref{e3a} with the same constants $c_0$ and $c_1$
and functions $f$ and $g$, and smooth functions $f_{1,n}$,
$g_{1,n}$, $f_{2,n}$, $g_{2,n}$, replacing $f_{1}$, $g_{1}$,
$f_{2}$, $g_{2}$, respectively, satisfying the $PK$ conditions
with the same constants, and $f_{i,n}\to f_i$ and $g_{i,n}\to
g_i$, $i=1,2$ pointwise a.e.
\end{theorem}

The next theorem establishes the same smoothness as above, for
all solutions to \eqref{e0}.

\begin{theorem}
\label{PK-aposteriori} Let $\A$ and $b$ satisfy the structure
conditions \eqref{e1.2b} and \eqref{e1}--\eqref{e3a} with
$f^2,\,f_1^2,\,f_2^2,\,g_2^2\in PK^{loc}$. Let $u$ be a weak
solution to \eqref{e0} in $\O_T$. There exists $\nu>0$ such
that, if for all $\O'\Subset\O$,
\[\lim\limits_{R\to0}
\sup\limits_{x\in\O'}\Big[W^f_p(x,R) +
W^{f_1^\frac{p}{p-1}}_p(x,R)+W^{g^p}_p(x,R)+W^{g_1^{p}}_p(x,R)\Big]<\nu,\]
then, for every $q\ge p$ and $l>0$, there exists $\beta>0$ such
that
\[
\n u\in L^\infty_{loc}\Big((0,T); L^{q-p+2}_{loc}(\O)\Big) \mbox{ and }
\n u\big(|\n u|-l\big)_+^{\frac q2-1}\in L^2_{loc}\Big((0,T); W^{1,2}_{loc}(\O)\Big)
\]
provided $g_2^2\in PK_\beta^{loc}$.  In particular, if
$g_2^2\in PK_0^{loc}$ then $\n u\in L^q_{loc}(Q)$ for every
$q<\infty$.
\end{theorem}

Finally, we give sufficient conditions for the boundedness of the gradient of solutions.

\begin{theorem}
\label{main} Let $\A$ and $b$ satisfy structure conditions
\eqref{e1.2b} and \eqref{e1}--\eqref{e3a}. Let $u$ be a weak
solution to \eqref{e0}. There exists $\nu>0$ such that, if for
all $\O'\Subset\O$,
\[\lim\limits_{R\to0}
\sup\limits_{x\in\O'}\Big[W^{g^p}_p(x,R)+W^{g_1^{p}}_p(x,R)\Big]<\nu,\]
and
\[
\lim\limits_{R\to0}
\sup\limits_{x\in\O'}\Big[
W_{\frac23,3}^{f^2}(x,R) + W_{\frac23,3}^{g^2}(x,R) + W_{\frac23,3}^{f_1^2}(x,R) +
W_{\frac23,3}^{g_1^2}(x,R) + W_{\frac23,3}^{f_2^2}(x,R) + W_{\frac23,3}^{g_2^2}(x,R)
\Big]<\nu,
\]
that
is, for any $\O'\Subset\O$,
\begin{equation}
\label{main-cond}
\lim_{R\to 0} \sup_{x\in \O'} \int_0^R \frac{dr}r\left(\frac1{r^{N-2}}\int_{B_r(x)}\left(f(y)^2+g(y)^2+f_1(y)^2
+g_1(y)^2+f_2(y)^2+g_2(y)^2\right)dy\right)^\frac12 <\nu,
\end{equation}
then
\[
\n u\in L^\infty_{loc}(\O_T),
\]
i.e. all solutions to \eqref{e0} are Lipschitz continuous with respect to the spatial variables.

In particular, if $g^p+g_1^p\in K_p$ and $f^2+g^2+f_1^2+g_1^2+f_2^2+g_2^2\in K_{\frac23,3}$, then every weak solution to~\eqref{e0} is locally Lipschitz.
\end{theorem}

Due to the scaling properties of equation~\eqref{e0} one can eliminate the smallness conditions on the coefficients $f,f_1$ and $f_2$.
The next statement though a simple consequence of the preceding theorem,  gives a generalization of the above result both in the sense of
the structure condition on the right hand side $b$ and on the conditions on the structure coefficients $f,f_1$ and $f_2$.

\begin{corollary}
\label{draft-cor}
Let $\A$ satisfy structure conditions \eqref{e1.2b} and $b$ satisfy the structure condition
\[
|b(x,t,u,z)|\le g(x)|z|^{p-1}+h(x)|u|^{p-1} + f(x).
\]
Let $u$ be a weak
solution to \eqref{e0}.
Assume that for every $\O'\Subset \O$,
\[
\sup_{x\in \O'}\int_0^R \frac{dr}{r} \left(\frac1{r^{N-2}}\int_{B_r(x)}[f(y)^2+f_1(y)^2+f_2(y)^2]dy\right)^\frac12 <\infty.
\]
Then there exists $\nu>0$ such that if for every $\O'\Subset \O$,
\[\lim\limits_{R\to0}
\sup\limits_{x\in\O'}\Big[W^{g^p}_p(x,R)+W^{g_1^{p}}_p(x,R)\Big]<\nu,\]
and
\[
\lim\limits_{R\to0}
\sup\limits_{x\in\O'}\Big[W^{h^2}_{\frac23,3}(x,R)+
 W_{\frac23,3}^{g^2}(x,R) +
W_{\frac23,3}^{g_1^2}(x,R) +  W_{\frac23,3}^{g_2^2}(x,R)
\Big]<\nu,
\]
then
$u$ is Lipschitz.
\end{corollary}
\begin{proof}
Let $\l>1$ to be chosen later.
Let
\[
\tau=\l^{p-2}, \quad v(x,\tau)=\l^{-1}u(x,t).
\]
Then $v$ satisfies the equation
\[
v_\tau - {\rm div}\,\widetilde\A(x,\tau,v,\n v)=\widetilde b,
\]
where $\big(\widetilde\A,\widetilde b\big)
(x,\tau,v,z)=\l^{1-p}\big(\A,b\big)(x,\l^{2-p}\tau,\l v , \l
z)$. 
For $\widetilde b$:
\[
|\widetilde b(x,\tau,v, \n v)|=\l^{1-p}|b(x,t,u,\n u)|\le g|\n v|^{p-1}+h|v|^{p-1}+\l^{1-p}f.
\]
Analogously one can verify the structure conditions on $\widetilde\A$:
\begin{eqnarray}
\label{ee1}
&&\lan (\part_z\widetilde\A) \mu,\mu\ran \ge c_0 |z|^{p-2}|\mu|^2, \quad \forall \mu,\,z\in\R^N,\\
\label{ee2}
&& |\part_z\widetilde\A|\le c_1 (|z|^{p-2}+\l^{2-p}),\\
\label{ee3}
&&|\part_u\widetilde\A| \le g_1(x)|z|^{p-2}+\l^{2-p}f_1(x),\\
\label{ee3a}
&&|\part_x\widetilde\A|\le g_2(x)|z|^{p-1}+ \l^{1-p}f_2(x),
\end{eqnarray}
So fix $\l$ large enough so that $\widetilde f=\l^{1-p}f$, $\widetilde f_1=\l^{2-p}f_1$, $\widetilde f_2=\l^{1-p}f_2$
satisfy the condition of Theorem~\ref{main} and the assertion follows.
\end{proof}

\begin{remark}
To compare Theorem~\ref{main} and Corollary~\ref{draft-cor}
with main results in \cite{DMl}, in which {\it elliptic} equations and systems are studied,
we first note that
the results in
\cite{DMl} concern the special case of the vector field $\A$ depends on $\n u$ only, that is when $f_1=g_1=f_2=g_2=0$ in structure conditions \eqref{e3},~\eqref{e3a}.
Theorem~1.4 in \cite{DMl} states only the existence of a locally  Lipschitz solution to the equation $-\div \A (\n u)=b(x,u,\n u)$ subject to the Dirichlet boundary
condition with boundary data from $W^{1,p}(\O)$ under the assumption of smallness of the Wolff type potential $W_{\frac23,3}$ of $f^2+g^2$.
%
%
The assertion that all the
solutions are Lipschitz is proved in Theorem~1.1 in \cite{DMl}, where  only  the particular case
$b(x,u,\n u)=f(x)$ is considered.
So both results follow from Corollary~\ref{draft-cor} as special cases.
\end{remark}

\begin{remark} As a consequence of Theorem~\ref{main} one can give sufficient conditions of the local boundedness of the gradient
of solutions to \eqref{e0} in terms of the structure
coefficients belonging to the Lorentz spaces. This is based on
the easily verifiable fact that $f\in L^{N,1}\Rightarrow f^2
\in K_{\frac23,3}$. We do not dwell upon this further, and
refer the reader to \cite{DMl} for an extensive discussion of
this point.
\end{remark}

\begin{remark}
\label{rem-dz}
In all the above results structure condition \eqref{e2} can be replaced  by a more general one $|\part_z\A|\le c_1 |z|^{p-2}+h_1(x)$ with the requirement
$h^2\in PK_0$ for Theorems~\ref{PK-existence},~\ref{PK-aposteriori}, and $h_1^2\in K_{\frac23,3}$ for Theorem~\ref{main} (compare this with $({\cal S}_3)$ in
\cite[Chap.\,VIII]{DiB}). We did not elaborate this further.
\end{remark}

\subsection{Auxiliary facts}

The following lemma provides an inequality of Hardy-type which is useful in the sequel.

\begin{lemma}
\label{lem-hardy}
Let $h\in W^{1,p}_{loc}(\O), h>0$. Suppose that $\Delta_p h\in
L^1_{loc}(\O)$ and $-\Delta_p h>0$. Then for any $\theta\in
\W^{1,p}(\O)$
\begin{equation}\label{Hardy}
    \int_{\O}\frac{(-\Delta_p h)}{h^{p-1}}|\theta|^p dx\le
\int_{\O} |\n \theta |^p dx.
\end{equation}

If in addition $h\in L^\infty(\O)$ then
\begin{equation}
\label{hardy} \int_{\O}{(-\Delta_p h)}|\theta|^p dx\le
\|h\|_\infty^{p-1}\int_{\O} |\n
\theta |^p dx.
\end{equation}
\end{lemma}
\proof First, by the Young inequality note that $pa^{p-1}b -
(p-1)a^p\le b^p$ for any $a,\,b>0$ and $p>1$. Let $\eps>0$ and
$0<\theta \in C_c^\infty(\O)$. Then it follows that
\[
 \n\left(\frac{\theta^p}{(h+\eps)^{p-1}}\right)|\n h|^{p-2}\n h
 \le p\frac{\theta^{p-1}|\n h|^{p-1}}{(h+\eps)^{p-1}}|\n \theta|-(p-1)\frac{\theta^p|\n h|^p}{(h+\eps)^p}
 \le |\n\theta|^p.
\]
Integrating the above and letting $\eps\to0$, we obtain
\eqref{Hardy} for $\theta\in C_0^\infty(\O)$. The general case
follows by approximation. In case $h\in L^\infty(\Omega)$, it
follows from \eqref{Hardy} that
\[
\int_{\O}{(-\Delta_p h)}|\theta|^p dx\le
\int_{\O}\frac{\|h\|_\infty^{p-1}}{h^{p-1}}(-\Delta_p h)|\theta|^p dx\le
\|h\|_\infty^{p-1}\int_{\O} |\n
\theta |^p dx.
\]
Hence \eqref{hardy} follows. \qed

\begin{lemma}
\label{lem-Dir} Let $f\ge 0, \, f\in L^1_{loc}$ and $u$ be the
weak solution to
\begin{equation}
\label{Dir}
-\Delta_p u=f \quad\text{in}\ B_{R}, \quad u|_{\partial B_{R}}=0.
\end{equation}
Then there exists $c>0$ such that
\[
\sup_{B_R}u(x)\le c\sup_{B_{R}}W^f_p(x,2R).
\]
\end{lemma}
\begin{proof}
Testing \eqref{Dir} by $u$ we obtain
\begin{equation}
\label{dual} \int_{B_{R}}|\n u|^p dy \le
\sup_{B_R}u(x) \int_{B_R} f(y)dy.
\end{equation}
By \cite[Theorem~4.8]{KiMa} (see also \cite[Theorem~2.125]{MZ}), for $x_0\in B_R$,
\begin{equation}
\label{KM}
u(x_0)\le c \left(\frac1{R^N}\int_{B_R(x_0)\cap B_R}u(y)^pdy\right)^\frac1p +c W^f_p(x_0, 2R).
\end{equation}

Using the Poincar\'{e} inequality, \eqref{dual}, the Young inequality and the definition of the Wolff potential we have
\begin{eqnarray}
\nonumber
&&\left(\frac1{R^N}\int_{B_{R}}u(y)^pdy\right)^\frac1p \le c\left(\frac1{R^{N-p}}\int_{B_{R}}|\n u(y)|^pdy\right)^\frac1p \\
\nonumber
&&\le c\left(\sup_{B_{R}}u(x)\right)^\frac1p \left(\frac1{R^{N-p}}\int_{B_{R}} f(y)dy\right)^\frac1p\le
\frac12 \sup_{B_{R}}u(x)+
c \left(\frac1{R^{N-p}}\int_{B_{R}(x_0)} f(y)dy\right)^\frac1{p-1}\\
\label{Wo1}
&&\le \frac12 \sup_{B_{R}}u(x)+cW^f_p(x_0,2R), \quad x_0\in B_R.
\end{eqnarray}
Combining \eqref{KM} and \eqref{Wo1} and taking supremum over $B_R$ we prove the assertion.
%
%
%
\end{proof}

As a consequence of Lemma~\ref{lem-hardy} and Lemma~\ref{lem-Dir} we obtain

\begin{corollary}
\label{lem-hardy-Kp} Let $\theta\in W^{1,p}_0(B_{R})$, $0\le
f\in L^1_{loc}$. Then there exists $\g>0$ such that
\begin{equation}
\label{hardy-Kp}
\int_{B_R}f\,|\theta|^p d x \le   \g \sup\limits_{B_{2R}}W^f_p(x,2R)^{p-1}\int_{B_R} |\n \theta |^p dx.
\end{equation}
\end{corollary}

\begin{corollary}
\label{lem-PK-Kp} Let $\theta\in W^{1,p}(B_{R})$, $0\le f\in
L^1_{loc}$. Then there exists $\g>0$ such that, for every
$\rho>R$,
\begin{equation}
\label{PK-Kp}
\int_{B_R}f\,|\theta|^p d x \le \g \sup\limits_{B_{2\rho}}W^f_p(x,2\rho)^{p-1}\left(\int_{B_\rho} |\n \theta |^p dx
+\tfrac1{(\rho-R)^p}\int_{B_{\rho}} |\theta |^p dx\right).
\end{equation}
\end{corollary}
\begin{proof}
Let $\xi\in C_c^1(B_{\rho})$ be such that $\xi = 1$ on $B_R$
and $|\n\xi|\le \frac2{\rho-R}$. Then, by~\eqref{hardy-Kp},
\[
\int_{B_R}f\,|\theta|^p d x \le \int_{B_\rho}f\,|\theta\xi|^p d x\le
\g \sup\limits_{B_{2\rho}}W^f_p(x,2\rho)^{p-1}\int_{B_\rho}\big|\n(\theta\xi)\big|^p d x.
\]
Hence the assertion follows.
\end{proof}
\medskip

The following proposition which is easy to verify, shows some
useful relations between the classes involved.

%

\begin{proposition}
\label{classes} Let $p,q>1$, $\alpha,\beta>0$. Assume that
either $\varkappa>\frac{\beta p}{\alpha q}\vee 1$ or
$1\le\varkappa=\frac{\beta p}{\alpha q}\le\frac{p-1}{q-1}$.
    Then there exists $c>0$ such that, for all $f>0$ and $R>0$,
\[W_{\alpha,q}^{f^\frac1\varkappa}(x,R) \le
cR^{\frac{\varkappa \alpha q - \beta p}{\varkappa(q-1)}}
\left(W_{\beta,p}^f(x,2R)\right)^{\frac{p-1}{\varkappa(q-1)}}.\]

In particular, for $p>2$, if
$\sup\limits_{x\in\O}W^{f^2}_{\frac23,3}(x,R)<\infty$ then
$|f|^q\in K_p$ for $q\in[1,2)$ and if $f^p\in K_p$ then $f^2\in
K_2\subset PK_0$.
\end{proposition}

\begin{proof}
It suffices to prove the first assertion. First observe that
there are constants $C\ge c>0$ dependent on $\beta, p$ and $N$
only such that, with $r_k=2^{-k}R$, $k=0,1,2,\ldots,$
\[
c\sum\limits_{k=1}^\infty \left(\frac1{r_k^{N-\beta p}}\int_{B_{r_k}(x)}f(y)dy\right)^\frac1{p-1}
\le W_{\beta,p}^f(x,R)
\le C\sum\limits_{k=0}^\infty \left(\frac1{r_k^{N-\beta p}}\int_{B_{r_k}(x)}f(y)dy\right)^\frac1{p-1}.
\]
Next, by the H\"{o}lder inequality,
\[
\left(\frac1{r^{N-\alpha q}}\int_{B_{r}(x)}f(y)^{\frac1\varkappa} dy\right)^\frac1{q-1}
\le c\left(\frac1{r^{N-\beta p}}\int_{B_{r}(x)}f(y) dy\right)^{\frac1{\varkappa(q-1)}}
r^{\frac{\varkappa \alpha q - \beta p}{\varkappa(q-1)}}.
\]
If $\varkappa(q-1)> p-1$ and  $\varkappa>\frac{\beta p}{\alpha
q}$ then, by the H\"{o}lder inequality,
\[
W_{\alpha,q}^{f^\frac1\varkappa}(x,R) \le cR^{\frac{\varkappa \alpha q - \beta p}{\varkappa(q-1)}}
\left(W_{\beta,p}^f(x,R)\right)^{\frac{p-1}{\varkappa(q-1)}}.
\]
If $\varkappa(q-1)\le p-1$ and  $\varkappa\ge\frac{\beta
p}{\alpha q}$  then
\[\begin{split}
W_{\alpha,q}^{f^\frac1\varkappa}(x,R) \le & cR^{\frac{\varkappa \alpha q - \beta p}{\varkappa(q-1)}}
\sum\limits_{k=0}^\infty \left(\frac1{r_k^{N-\beta p}}\int_{B_{r_k}(x)}f(y) dy\right)^{\frac1{p-1}}
\sup\limits_k \left(\frac1{r_k^{N-\beta p}}\int_{B_{r_k}(x)}f(y) dy\right)^{\frac1{\varkappa(q-1)}-\frac1{p-1}}
\\
\le & cR^{\frac{\varkappa \alpha q - \beta p}{\varkappa(q-1)}}
\left(W_{\beta,p}^f(x,2R)\right)^{\frac{p-1}{\varkappa(q-1)}}.
\end{split}\]
The inclusion $K_2\subset PK_0$ is well known in the
standard theory of Kato classes \cite{Bsimon}.
\end{proof}

\section{Proof of Theorem~\ref{PK-existence}}

 We start with constructing  an appropriate local approximation of equation \eqref{e0} and obtaining a priori estimates.
\paragraph{Approximation.}
For $\eps>0$ let $j_\eps$ be the standard mollifier in $\R^N$.
Denote $\A_\eps =\A*j_\eps+\eps z$, smoothing with respect to
$x$ variable only. For $b$ we introduce $b_\eps=b\wedge
\frac1\eps \vee (-\frac1\eps)$. Also set $f_{1,\eps}=
f_1*j_\eps$, $f_{2,\eps}= f_2*j_\eps$,
$g_{1,\eps}=g_1*j_\eps$ and $g_{2,\eps}=g_2*j_\eps$.
Note that the structure conditions \eqref{e1.2b} and
\eqref{e1}--\eqref{e3a} hold with $\A_\eps, b_\eps, f_{1,\eps},
f_{2,\eps}, g_{1,\eps} $ and $g_{2,\eps}$ replacing $\A,b,f_1,f_2,g_1$ and $g_2$, respectively.
Note also that, if $F\in PK_0$ then $F*j_\eps\in PK_0$ for all $\eps>0$, with the same function $C(\beta)$.

Let $Q$ denote a cylinder $B_R\times (t_1,t_2)$ such that
$Q\Subset \O_T$. Consider the following approximating equation.
\begin{equation}
\label{e-approx} u_t-\div \A_\eps (x,t,u, \n u)=b_\eps(x,t,u,\n u) ,\quad (x,t)\in
Q.
\end{equation}

In the rest of this subsection we study solutions to
\eqref{e-approx} in $Q$.  Our task in the sequel is to obtain
estimates which are uniform in $\eps$ and which will allow us
to pass to the limit $\eps\to 0$.

In order to simplify the notation, in the rest of this
subsection in all proofs we drop subindex $\eps$. We often use
the Steklov averaging $T_h$, $h>0$, defined by
\[
\big(T_h v\big)(x,t)=\frac1{2h}\int_{-h}^h v(x,t+s)ds.
\]
We write $v_h=T_h v$.

\begin{proposition}\label{Dt}
Let $u_\eps\in V(Q)$ be a solution to \eqref{e-approx} in $Q$.
Then, for every $Q'=B'\times(t_1',t_2')\Subset Q''=
B''\times(t_1'',t_2'')\Subset Q$, there exists $\g>0$
independent of $\eps$ such that
\[
\|\partial_t u_\eps\|_{L^{p'}\big((t_1',t_2');\; W^{-1,p'}(B')\big)}\le \g
\left(\|\n u_\eps\|_{L^p(Q'')} + \|u_\eps\|_{L^p(Q'')} + \|f\|_{L^1(Q')}^{\frac1{p'}}\right).
\]
\end{proposition}
\begin{proof}
To prove the assertion it suffices to show that, for all $\xi\in
C_c^1(Q')$,
\[
\left|\iint_{Q'}u\partial_t\xi dx\,dt\right|\le \g\|\n\xi\|_{L^p(Q')}.
\]
It follows from~\eqref{e-approx} and structure
conditions~\eqref{e1.2b} that
\[
\begin{split}
\left|\iint_{Q'}u\partial_t\xi dx\,dt\right|&\le
\iint_{Q'}\left[|\A(u,\n u)|\,|\n\xi| + |b(u,\n u)|\,|\xi|\right]dx\,dt
\\
&\le \g\|\n u\|_{L^p(Q')}\left(\|\n\xi\|_{L^p(Q')} + \|g\xi\|_{L^p(Q')}\right)\\
& + \g\|\n\xi\|_{L^1(Q')} + \g\|f^{\frac1p}u\|_{L^p(Q')}^{p-1}\|f^{\frac1p}\xi\|_{L^p(Q')}
+ \g \|f\|_{L^1(Q')}^{\frac1{p'}}\|f^{\frac1p}\xi\|_{L^p(Q')}.
\end{split}
\]
Corollary~\ref{lem-hardy-Kp} implies that
\[
\|f^{\frac1p}\xi\|_{L^p(Q')} + \|g\xi\|_{L^p(Q')}\le \g\|\n\xi\|_{L^p(Q')}.
\]
Finally,
from Corollary~\ref{lem-PK-Kp}
we conclude that
\[
\|f^{\frac1p}u\|_{L^p(Q')}^{p-1} \le \g \left(\|\n u\|_{L^p(Q'')} + \|u\|_{L^p(Q'')}\right).
\]
Hence the assertion follows.
\end{proof}

\begin{proposition}\label{appr-reg}
Let $u_\eps\in V(Q)$ be a solution to \eqref{e-approx} in $Q$.
Then $u_\eps\in L^2_{loc}\big((t_1,t_2);W^{2,2}_{loc}(Q)\big)$
and $|\n u_\eps|^q\n u_\eps \in
L^\infty_{loc}\big((t_1,t_2);L^{2}_{loc}(Q)\big)\cap
L^2_{loc}\big((t_1,t_2);W^{1,2}_{loc}(Q)\big)$ for all $q\ge
\frac p2$.
\end{proposition}
\begin{proof}
The assertion follows by the direct approach via finite
differences (see, e.g.  \cite[Section~VIII.3]{DiB} and
\cite[Section~IV.5]{La}).
\end{proof}

The main result of this subsection is the following a priori
estimate.

\begin{proposition}
\label{apriori1} Let $Q'=B'\times(t_1',t_2')\Subset
Q''=B''\times(t_1'',t_2'')\Subset Q=B\times(t_1,t_2)$ be
cylinders compactly embedded in $\O_T$ and let $u_\eps\in V(Q)$
be a solution to \eqref{e-approx} in $Q$. Assume that $\A$ and
$b$ satisfy the structure conditions \eqref{e1.2b} and
\eqref{e1}--\eqref{e3a} with the functions
$(f^2+f_1^2+g_1^2+f_2^2+g_2^2)\ind_{B_R}\in PK$ 
and that there exists $M$ independent of $\eps$ such that
$|u_\eps|\le M$ on $Q''$. Then, for every $l>0$ and $\a\ge0$,
there exist constants $\beta$ and $\g$ independent of $\eps$,
such that, if $(g^2+g_1^2+g_2^2)\ind_{B_R}\in PK_\beta$ then
\begin{equation}
\label{ep2.8}
\begin{split}
\esssup_{t\in[t_1',t_2']}\int_{B'} |\n u_\eps|^{2+2\a}dx
 & +\iint_{Q'} \Big|D\big(\n u_\eps(|\n u_\eps|-l)_+^{\a+\frac p2-1}\big)\Big|^2dxd\tau\\
\le & \g\left(\iint_{Q''} \left(|\n u_\eps|^{p} + F^2 +1\right)dxd\tau\right)^{\a+1}
 + \g\left(\iint_{Q''} \left(F^2 +1\right)dxd\tau\right)^{\frac N{N+2}},
\end{split}
\end{equation}
with $F=f+f_1+g +g_1+ f_2+g_2$.
\end{proposition}
The proof of this proposition is divided into several lemmas,
some of which will be used in further argument as well.

Since $u_\eps$ is twice  weakly differentiable, we can differentiate equation~\eqref{e-approx}. This is done in the next lemma.
\begin{lemma}
\label{lem2} Let $Q$ be as in Proposition~\ref{apriori1} and
$u_\eps\in V(Q)$ be a weak solution to \eqref{e-approx} in $Q$.
Then for every $\zeta \in H^1_c(Q\to \R^N)$ and for all
$t_1<t_1'<t_2'<t_2$, one has
\begin{eqnarray}
\nonumber
\int_{B_R}\lan \n u_\eps,\zeta\ran dx\Big|_{t_1'}^{t_2'}-
\int\limits_{t_1'}^{t_2'}\int\limits_{B_R}
\lan \n u_\eps, \part_t\zeta\ran dx\,dt +
\int\limits_{t_1'}^{t_2'}\int\limits_{B_R} \TR{(D\zeta)(\part_z\A_\eps)D^2u_\eps } dx\,dt\\
\label{e5}
=- \int\limits_{t_1'}^{t_2'}\int\limits_{B_R}
[\TR{(D\zeta)\part_x\A_\eps}+\lan (D\zeta)(\part_u\A_\eps),\n u_\eps\ran + b_\eps\, {\rm div}\zeta]dx\,dt.
\end{eqnarray}
\end{lemma}

\begin{proof}
First, let $\zeta\in C_c^2(Q\to \R^N)$. Test
equation~\eqref{e-approx} by $-\div \zeta$. Integrating by
parts we obtain
\[
\int_{B_R}\lan \n u_\eps,\zeta\ran dx\Big|_{t_1'}^{t_2'}- \int\limits_{t_1'}^{t_2'}\int\limits_{B_R}
\lan \n u, \part_t \zeta \ran dxdt+ \int\limits_{t_1'}^{t_2'}\int\limits_{B_R}\TR{(D\zeta)(D\A)}dxdt
=-\int\limits_{t_1'}^{t_2'}\int\limits_{B_R} b\,\div\zeta dxdt.
\]
Observing that
\[
D\A=(\part_z\A)D^2u +\part_u\A \otimes \n u +\part_x \A
\]
we arrive at \eqref{e5}. The general case follows by approximation.
\end{proof}

\bigskip

\begin{remark}
\label{rem1} Note that \eqref{e1} implies that, for $M\in
\R^{N\times N}$, one has
\[
\TR{ M^T(\part_z\A)M}\ge c_0|z|^{p-2}|M|^2.
\]

Indeed, let $M=\{m_{kl}\}$. Then
\[
\begin{split}
\TR{M^T(\part_z\A)M} = \sum_{jkl} m_{kj}(\partial_{z_l}\A_k)m_{lj}=\sum_j\sum_{kl}(\partial_{z_l} \A_k)m_{lj}m_{kj}\\
\ge \sum_jc_0|z|^{p-2}\sum_k m_{kj}^2=c_0|z|^{p-2}|M|^2.
\end{split}
\]
\end{remark}

The proof of Proposition~\ref{apriori1} is performed  by a Moser-type
iteration with finite number of steps. The following two lemmas contain the main technical
part of the proof.

\begin{lemma}
\label{lem3} Let $Q'$, $Q$ and
 $u_\eps$ be as in Proposition~\ref{apriori1}. Let $\xi\ge 0$, $\xi\in C^\infty(Q')$
vanishing on the parabolic boundary $\mathcal{P}Q'$. Let
$\Phi\in C^{0,1}_b(\mathbb{R})$, $\Phi(0)=0$ and
$\mathcal{G}(s):=\int_0^s\tau \Phi(\tau)d\tau$. Let $\zeta:=\n
u_\eps \Phi(|\n u_\eps|)\xi$.
Then for almost all (a.a.) $\tau\in(t_1',t_2')$ one has
\[
\begin{split}
&\int\limits_{B'} \mathcal{G}(|\n u_\eps(\tau)|)\xi(\tau)dx
+ \int\limits_{t_1}^\tau \int\limits_{B'} \TR{(D\zeta)(\part_z\A_\eps)D^2u_\eps } dx\,dt
\\
& \leq  \int\limits_{t_1}^\tau \int\limits_{B'} \mathcal{G}(|\n u_\eps|)\part_t\xi dx\,dt
- \int\limits_{t_1}^\tau\int\limits_B
[\TR{(D\zeta)\part_x\A_\eps}+\lan (D\zeta)(\part_u\A_\eps),\n u_\eps\ran + b_\eps\, {\rm div}\zeta]dx\,dt.
\end{split}
\]

\end{lemma}
\begin{proof}
As before, we write $u_h=T_hu$, where $T_h$ is the Steklov
averaging with $h<\min\{t_2-t_2',t_1'-t_1\}$. With notation
above set $\zeta_h=\n u_h \Phi(|\n u_h|)\xi$. We apply
$T_h\zeta$ as the test vector function in \eqref{e5}:
\[
\begin{split}
\int_{B_R}\lan \n u_h(\tau),\zeta_h(\tau)\ran dx
& -\int\limits_{t_1}^\tau\int\limits_{B_R}
\lan \n u_h, \part_t\zeta_h\ran dx\,dt +
\int\limits_{t_1}^\tau\int\limits_{B_R} \TR{(D\zeta_h)T_h[(\part_z\A)D^2u] } dx\,dt\\
= & - \int\limits_{t_1}^\tau\int\limits_{B_R}
[\TR{(D\zeta_h)T_h\part_x\A}+\lan (DT_h\zeta_h)(\part_u\A),\n u\ran + T_h[b]\, {\rm div}\zeta_h]dx\,dt.
\end{split}
\]
Now we pass to the limit as $h\to0$.
For the first two terms in the left hand side we have
\[
\begin{split}
\int_{B_R}\lan \n u_h(\tau),\zeta_h(\tau)\ran dx - \int\limits_{t_1}^\tau\int\limits_{B_R}\lan \n u_h, \part_t\zeta_h\ran dx\,dt
= \int\limits_{t_1}^\tau\int\limits_{B_R} \lan \part_t\n u_h,\n u_h\ran \Phi(|\n u_h|) \xi dx\,dt\\
=\frac12\int\limits_{t_1}^\tau\int\limits_{B_R}  (\part_t|\n u_h|^2) \Phi(|\n u_h|) \xi dx\,dt
=\int\limits_{t_1}^\tau\int\limits_{B_R}\left(\part_t\mathcal{G}(|\n u_h|)\right)\xi dx\,dt\\
=\int\limits_{B_R} \mathcal{G}(|\n u_h(\tau)|)\xi(\tau) dx
- \int\limits_{t_1}^\tau\int\limits_{B_R} \mathcal{G}(|\n u_h|)\partial_t\xi dx\,dt.
\end{split}
\]
Since $\n u_h \to \n u$ a.e. as $h\to 0$, we obtain
\[
\begin{split}
\lim_{h\to 0}\int\limits_{t_1}^\tau\int\limits_{B_R} \mathcal{G}(|\n u_h|)\partial_t\xi dx\,dt =
\int\limits_{t_1}^\tau\int\limits_{B_R} \mathcal{G}(|\n u|)\partial_t\xi dx\,dt,
 \\
\liminf_{h\to 0} \int\limits_{B_R} \mathcal{G}(|\n u_h|(\tau))\xi(\tau)dx\ge \int\limits_{B_R} \mathcal{G}(|\n u|(\tau))\xi(\tau)dx,
\end{split}
\]
and the assertion follows.
\end{proof}


\begin{lemma}
\label{step1} Let $Q'\Subset Q''\Subset Q$, \ $u_\eps$ and $M$
be as in Proposition~\ref{apriori1}. For $l>0$ and $\a\ge0$,
let
\[
\Phi_\a(s)=(s-l)_+^{2+2\a}s^{-2},\ \a\ge 0,  \quad \Ge_\a(s)=\int_0^s r\Phi_\a(r) dr.
\]
Assume that 
$(f^2+g^2+f_1^2+g_1^2+f_2^2+g_2^2)\ind_{B_R}\in PK$. Then, for
every $\a\ge 0$ and $l>0$ there exist $\beta$ and $\g$
independent of $\eps$ such that
\begin{eqnarray}
\nonumber
\esssup_t \int_{B'} \Ge_\a(|\n u_\eps(t)|)dx + \iint_{Q'} |\n u_\eps|^{p-2} |D^2 u_\eps|^2 \Phi_\a(|\n u_\eps|)dx\,dt\\
\nonumber
+\iint _{Q''} |\n u_\eps|^{p-1} |\n |\n u_\eps||^2 \Phi'_\a(\n u_\eps)dx\,dt\\
\label{e-PK-lem}
\le \g \iint_{Q''} (f^2+g^2+f_1^2+g_1^2 +f_2^2+g_2^2 ) dx\,dt
+\g \iint_{Q''} |\n u_\eps|^p \Phi_\a(|\n u_\eps|) dx\,dt
\end{eqnarray}
provided $(g^2+g_1^2+g_2^2)\ind_{B_R}\in PK_\beta$.
\end{lemma}

\begin{proof}
Since $|u|\le M$ on $Q''$, it follows that $|b(u, \n u)|\le
g|\n u|^{p-1} + \gamma f$ on $Q''$. 

In the rest of the proof we omit the subscript $\a$ in
$\Phi_\a$ and $\Ge_\a$. Let $\xi$ be the standard cut-off
function vanishing on the parabolic boundary of $Q''$, which is
equal to 1 on $Q'$.

By Lemma~\ref{lem3} with $\zeta= \Phi(|\n u|)\xi^2\n u$ as a
test function we have
\begin{equation}
\label{e1a}
\begin{split}
&\sup\limits_t\int\limits_{B''} \Ge(|\n u|(t) \xi^2dx +\iint\limits_{Q''} \TR{(D\zeta)(\part_z\A)D^2u }dx\,dt\\
&\le 2 \iint\limits_{Q''}  \Ge(|\n u|)\xi\part_t\xi dx\,dt +
\iint\limits_{Q''} (|\part_x\A|+|\part_u \A||\n u| +|b|)|D\zeta|dx\,dt.
\end{split}
\end{equation}
Note that
\[
D\zeta = \Phi(|\n u|)\xi^2 D^2u + \Phi'(|\n u|)\xi^2 \n u \otimes \n|\n u| + 2\Phi(|\n u|)\xi\n u \otimes \n\xi.
\]
Now we estimate the left hand side of \eqref{e1a} from below using \eqref{e1}--\eqref{e2}, Remark~\ref{rem1}, the identities $D^2u\n u =|\n u|\n|\n u|$ and
$s\Phi'(s) = 2\a\Phi(s) + (2+2\a)l(s-l)^{1+2\a}s^{-2}$:
\[
\begin{split}
\TR{(D\zeta)(\part_z\A)D^2u } \ge & c_0\Phi(|\n u|)|\n u|^{p-2}\big|D^2u\big|^2\xi^2
+ c_0\Phi'(|\n u|)|\n u|^{p-1}\big|\n|\n u|\big|^2\xi^2
\\
& - 2(1+l^{2-p})c_1\Phi(|\n u|)|\n u|^{p-1}\big|\n|\n u|\big|\xi|\n\xi|
\\
\ge & c_0\Phi(|\n u|)|\n u|^{p-2}\big|D^2u\big|^2\xi^2 + \tfrac12c_0\Phi'(|\n u|)|\n u|^{p-1}\big|\n|\n u|\big|^2\xi^2
\\
&- c_{l,p,\a}\Phi(|\n u|)|\n u|^{p}|\n\xi|^2.
\end{split}
\]

The first term on the right hand side of \eqref{e1a} is
estimated using the elementary inequality $\Ge(s)\le c_\a
l^{2-p} \Phi(s)$. The second term is estimated by the Schwartz
inequality using the $PK_0$ condition. To shorten the
exposition, we denote $F_\eps:=f+f_{1,\eps}+f_{2,\eps}$,
$G_\eps:=g + g_{1,\eps}+g_{2,\eps}$. Observe that $F_\eps^2$
and $G_\eps^2$ belong to $PK$ with the same constants as $F^2$
and $G^2$, respectively.

It follows from \eqref{e1.2b} and \eqref{e3}--\eqref{e3a} that
\[|\part_x\A|+|\part_u \A||\n u| +|b| \le |\n u|^{p-1}G_\eps +
|\n u|f_{1,\eps} + f+f_2.\] To estimate the right hand side of
\eqref{e1a} we use
the Schwartz inequality and the estimate $\Phi(s)s^p +
\Phi'(s)s^{p+1}\le c_l\big(\Phi(s)(s-l)_+^p +1\big)$ in order
to conclude that,
for all $\d>0$,
\[
\begin{split}
\Phi(|\n u|)\xi^2 \big|D^2u\big|\,|\n u|^{p-1}G_\eps\le & \d\Phi(|\n u|)|\n u|^{p-2}\big|D^2u\big|^2\xi^2
+ \tfrac{c_l}{\d}\big(\Phi(|\n u|)(|\n u|-l)_+^p +1\big) \xi^2G_\eps^2;
\\
\Phi'(|\n u|)\xi^2 |\n u|^p\big| \n|\n u|\big|G_\eps \le & \d\Phi'(|\n u|)|\n u|^{p-1}\big|\n|\n u|\big|^2\xi^2
+  \tfrac{c_l}{\d}\big(\Phi(|\n u|)(|\n u|-l)_+^p +1\big) \xi^2G_\eps^2;
\\
\Phi(|\n u|)\xi |\n u|^p|\n\xi|G_\eps\le & \tfrac12\big(\Phi(|\n u|)(|\n u|-l)_+^p +1\big) \xi^2G_\eps^2
+ \tfrac12\Phi(|\n u|)|\n u|^{p}|\n\xi|^2.
\end{split}
\]
Similarly, since $p>2$, for every $\sigma>0$ there exists
$c_{l,\alpha,\sigma}>0$ such that
\[
(\Phi(s)s^{2-p} + \Phi'(s)s^{3-p})(1+s^2)\le \sigma \Phi(s)(s-l)_+^p + c_{l,\alpha,\sigma}.
\]
Hence we conclude that, for all $\d>0$, there exist
$c_{l,\alpha,\delta}>0$ such that
\[
\begin{split}
&\Phi(|\n u|)\xi^2 \big|D^2u\big|(|\n u|f_{1,\eps}+f+f_{2,\eps})
\\
&\le   \d\Phi(|\n u|)|\n u|^{p-2}\big|D^2u\big|^2\xi^2
+ \tfrac{c}{\d}\Phi(|\n u|)|\n u|^{2-p}(|\n u|f_{1,\eps} + f + f_{2,\eps})^2 \xi^2
\\
&\le  \d\Phi(|\n u|)|\n u|^{p-2}\big|D^2u\big|^2\xi^2 + \d\Phi(|\n u|)(|\n u|-l)_+^p\xi^2F_\eps^2 +
c_{l,\alpha,\delta}\xi^2F_\eps^2;
\\
&\Phi'(|\n u|)\xi^2 |\n u|\big| \n|\n u|\big| (|\n u|f_{1,\eps}+f+f_{2,\eps})
\\
&\le
\d\Phi'(|\n u|)|\n u|^{p-1}\big|\n|\n u|\big|^2\xi^2
+  \tfrac{c}{\d}\Phi'(|\n u|)|\n u|^{3-p}(|\n u|f_{1,\eps} + f + f_{2,\eps})^2 \xi^2
\\
&\le  \d\Phi'(|\n u|)|\n u|^{p-1}\big|\n|\n u|\big|^2\xi^2 + \d\Phi(|\n u|)(|\n u|-l)_+^p\xi^2F_\eps^2 +
c_{l,\alpha,\delta}\xi^2F_\eps^2;
\\
&\Phi(|\n u|)\xi |\n u||\n\xi|(|\n u|f_{1,\eps}+f+f_{2,\eps})
\\
&\le
\tfrac12\Phi(|\n u|)|\n u|^{p}|\n\xi|^2
+
\tfrac12\Phi(|\n u|)|\n u|^{2-p}(|\n u|f_{1,\eps} + f + f_{2,\eps})^2\xi^2
\\
&\le  \tfrac12\Phi(|\n u|)|\n u|^{p}|\n\xi|^2 + \d\Phi(|\n u|)(|\n u|-l)_+^p\xi^2F_\eps^2 +
c_{l,\alpha,\delta}\xi^2F_\eps^2.
\end{split}
\]

 To complete the proof it remains to estimate the
term $\iint_{Q''}\Phi(|\n u|)(|\n u|-l)_+^p \xi^2
(F_\eps^2+G_\eps^2)dx\,dt$,
which is done by the direct use of the $PK$ condition noting
the inequality $\big|\n \big(\sqrt{\Phi(|\n u|)}(|\n u
|-l)_+^{\frac p2}\big)\big|^2 \le c \Phi(|\n u|)(|\n u
|-l)_+^{p-2} |\n |\n u||^2$. We omit further details.
\end{proof}


\begin{proof}[Proof of Proposition~\ref{apriori1}]
To prove the proposition it suffices to show that, for $\a>0$ and a cylinder $Q_1$ such that $Q'\Subset Q_1\Subset Q''$,
\begin{equation}\label{intermed}
    \iint_{Q_1} |\n u|^{p+2\a}dx\,dt\le \g_\a\left(\iint_{Q''} \left(|\n u|^{p} + F^2 +1\right)dx\,dt\right)^{\a+1}
+ \g_\a\left(\iint_{Q''} \left(F^2 +1\right)dx\,dt\right)^{N/{(N+2)}}.
\end{equation}
Then the assertion follows from Lemmma~\ref{step1}.

The proof of \eqref{intermed}
follows the line of the argument from \cite[Ch.VIII, Lemma~4.1]{DiB}.
We will iterate with respect to $\a$ as it is done in
\cite[p.232--233]{DiB} (with $\b$ in place of our $2\a$).
Let $Q^\dag=(t_1^\dag,t_2^\dag)\times B^\dag$,
$Q^\ddag=(t_1^\ddag,t_2^\ddag)\times B^\ddag$ be such that
$Q^\dag\Subset Q^\ddag\Subset Q$.
Fix $\a>0$. Let
$\Phi_\a$ and $\Ge_\a$ be as in Lemma~\ref{step1} with $l=1$,
$\Psi_\a(s)=\int_1^s r^{p/2-1}\sqrt{\Phi_\a(r)}dr$ with . Note
that $\Psi_\a(s)\le s^{p/2+\a}$ and
$\big|\Psi_\a'(s)\big|^2=s^{p-2}\Phi_\a(s)$. Using the
definitions of $\Phi_\a$, $\Psi_\a$ and $\Ge_\a$ and the
Sobolev inequality we obtain
\begin{eqnarray*}
\iint_{Q^\dag} |\n u|^{p+\frac 4N
+2\a\left(1+\frac 2N\right)}dx\,dt\le
2^{p+\frac 4N +2\a\left(1+\frac 2N\right)}|Q^\dag|
+ \iint_{Q^\dag\cap\{|\n u|>1\}}\Psi_\a^2(|\n u|)\Ge_\a^{2/N}(|\n u|)dx\,dt\\
\le \g|Q^\dag|+\g \iint_{Q^\dag}\left(|\n\Psi_a(|\n u|)|^2 + \Psi_\a^2(|\n u|)\right)dx\,dt\left(\sup_t \int_{B^\dag} \Ge_\a(|\n u|)\xi^2 dx\right)^{2/N}.
\end{eqnarray*}
By Lemma~\ref{step1} we estimate the right hand side of the
above inequality, which gives
\begin{equation}\label{step2}
    \iint_{Q^\dag} |\n u|^{p+\frac 4N +2\a\left(1+\frac 2N\right)}dx\,dt\le
\g |Q^\dag|+
\g \left(\iint_{Q^\ddag} (F^2+ |\n u|^{p+2\a})dx\,dt\right)^{1+2/N}.
\end{equation}
Consider the exhaustion of $Q''$ by cylinders $Q_0=Q'\Subset Q_1 \Subset Q_2\Subset \dots \Subset Q_n\Subset \dots \Subset Q''$.
By iterating \eqref{step2} with $\vark_n:=(1+\tfrac2N)^{n-1}$,  we obtain
\begin{equation}\label{special}
    \iint_{Q_1} |\n u|^{p+2\vark_n-2}dx\,dt\le
    \g_n\left(\iint_{Q_n} \left(|\n u|^{p} + F^2 +1\right)dx\,dt\right)^{\vark_n}
+ \g_n\iint_{Q_n} \left(F^2 +1\right)dx\,dt .
\end{equation}
This proves~\eqref{intermed} for $\a=\vark_n-1$, $n\in \mathbb{N}$. For a general $\a>0$ fix $n$ such that $\vark_n>\a+1>\vark_{n-1}$.
Then there exists $s\in(0,1)$ such that $p+2\a = s(p+2\vark_n-2) + (1-s)p$. Then
\[
\iint_{Q_1} |\n u|^{p+2\a}dx\,dt \le \left(\iint_{Q_1} |\n u|^{p+2\vark_n-2}dx\,dt\right)^s
\left(\iint_{Q_1} |\n u|^{p}dx\,dt\right)^{1-s}.
\]
Now~\eqref{intermed} follows from~\eqref{special} and the Young inequality.

\end{proof}

The following a priori estimate, mainly extracted from
\cite[Theorem 1.1]{LS2}, is a ground for the assumption in
Proposition~\ref{apriori1} that $u_\eps$ is locally bounded
uniformly in $\eps$.

\begin{proposition}\label{max_u}
Let $u_\eps\in V(Q)$ be a solution to \eqref{e-approx} in $Q$.
Then, for every $Q'=B'\times(t_1',t_2')\Subset
Q''=B''\times(t_1'',t_2'')\Subset Q$ and
$\rho<\frac14\min[1,dist(B',\part B),\sqrt{t_1'-t_1''},
\sqrt{t_2''-t_2'}]$, there exists $\g_\rho>0$ independent of
$\eps$ such that
\[
\sup\limits_{Q'}|u_\eps|\le \g_\rho\left(\iint_{Q''}|u_\eps|^{p+\frac1{Np'}}dx\,dt\right)^{\frac{pN}{2pN+p-1}}
+ \g_\rho\sup\limits_{t\in(t_1'',t_2'')}\left(\int_{B''}|u_\eps|^2dx\right)^{\frac12}
+ \g_\rho\sup\limits_{x\in B''}W^{g^p+f}(x,2\rho) + \g_\rho.
\]
\end{proposition}
\begin{proof}
The fact that $u\in L^\infty_{loc}(Q)$ is established in
\cite[Theorem 1.1]{LS2}. The actual estimate follows from
\cite[(3.18)]{LS2}.
\end{proof}

The next proposition establishes the existence of a solution to a
initial-boundary value problem for~\eqref{e-approx}.

\begin{proposition}
\label{approx}
 Let $v\in V(\O_T)\cap
L^{p'}_{loc}\big((0,T);\;W^{-1,p'}_{loc}(\O)\big)$. Then there
exists a solution \\ $u_\eps\in
L^p\big((t_1,t_2);W^{1,p}(B_R)\big)$ to \eqref{e-approx} on $Q$
subject to the condition $u_\eps|_{\P Q}=v$, where $\P Q$ is
the parabolic boundary of $Q$.
\end{proposition}
\begin{proof}
The assertion follows from \cite{Li}.
\end{proof}

The following proposition establishes first a posteriori
estimates for a solution to an initial-boundary value problem
for~\eqref{e-approx}.

\begin{proposition}\label{funct_int_est}
Let $u_\eps$ be a weak solution to \eqref{e-approx} in
$Q\Subset \O_T$. Assume that there exists $v\in V(\O_T)\cap
L^{p'}_{loc}\big((0,T);\;W^{-1,p'}_{loc}(\O)\big)$ such that
$u_\eps(x,t)=v(x,t)$ on $\P Q$. Then, provided
$\sup\limits_{x\in {B_R}}W_p^{g^p+f}(x,2R)$ is small enough,
the following estimates hold: there exists $\g$ independent of
$\eps$ such that
\[
\begin{split}
\sup_{\tau\in (t_1,t_2)}\int_{B_R}  u^2(\tau) dx + \iint_Q |\n
u_\eps|^pdx\,dt \le \g\sup_{\tau\in (t_1,t_2)}\int_{B_R}v^2(\tau)
dx + \g \iint_Q \big(|\n v\big|^p+f|v|^p+f\big)dx\,dt
\\
+\g\left\|\partial_tv\right\|_{L^{p'}\big((t_1,t_2);\;W^{-1,p'}(B_R)\big)},\\
\iint_Q|u|^{p+\frac{2p}N}dx\,dt \le
\g\iint_Q|v|^{p+\frac{2p}N}dx + \g\Big(\sup_{\tau\in (t_1,t_2)}\int_{B_R}v^2(\tau)
dx + \iint_Q \big(|\n v\big|^p+f|v|^p+f\big)dx\,dt
\\+
\left\|\partial_tv\right\|_{L^{p'}\big((t_1,t_2);\;W^{-1,p'}(B_R)\big)}\Big)^{1+\frac pN}.
\end{split}
\]
\end{proposition}

\begin{proof} Fix $(s,\tau)\Subset(t_1,t_2)$. Test
\eqref{e-approx} by $\xi=T_h(u_h - v_h)$ with
$h<\min\{t_2-\tau, s-t_1\}$, to obtain that
\[
\begin{split}
\frac12 \int_{B_R}[(u_h(\tau)-v_h(\tau))^2- &(u_h(s)-v_h(s))^2]dx  +
\int\limits_s^\tau\int\limits_{B_R} \<T_h\big[\A(u,\n u)\big], \n(u_h-v_h)\> dx\,dt\\
= & \int\limits_s^\tau\int\limits_{B_R} T_h[b(u,\n u)] (u_h-v_h) dx\,dt
+ \int\limits_s^\tau\int\limits_{B_R} (u_h-v_h) (\partial_tv_h)dx\,dt.
\end{split}
\]
Note that
\[
\left|\int\limits_s^\tau\int\limits_{B_R} (u_h-v_h) (\partial_tv_h)dx\,dt\right|
\le \left(\int\limits_s^\tau\int\limits_{B_R} |\n(u_h-v_h)|^pdx\,dt\right)^{\frac1p}
\left\|\partial_tv_h\right\|_{L^{p'}\big((s,\tau);\;W^{-1,p'}(B_R)\big)}.
\]
So we can pass to the limit as $h\to 0$ and then to the limit as $s\to t_1$
to obtain that
\[
\begin{split}
&\tfrac12\sup_{\tau\in (t_1,t_2)}\int_{B_R}(u(\tau)-v(\tau))^2dx +
\iint_Q\<\A(u,\n u),\n u\>  dx\,dt \\
&\le \iint_Q[|\A(u,\n u)|\,|\n v| + |b(u,\n u)|\,|u-v|]dx\,dt
 + \left(\iint_Q |\n(u-v)|^pdx\,dt\right)^{\frac1p}
\left\|\partial_tv\right\|_{L^{p'}((t_1,t_2);\;W^{-1,p'}(B_R))}.
\end{split}
\]
Then, using  structure conditions \eqref{e1.2b} and the Young
inequality 
we obtain that, for all $\d>0$ there exists $\g>0$ such that
\[
\begin{split}
\sup_{\tau\in (t_1,t_2)}\int_{B_R}\big(u(\tau)-v(\tau)\big)^2 dx & +
\iint_Q |\n u|^pdx\,dt
\le \d \iint_Q \big|\n (u-v)\big|^pdx\,dt + \g \iint_Q g^p|u-v|^p dx\,dt\\
&+\g \iint_Q f|u-v|^p dx\,dt + \g\iint_Q f|v|^pdx\,dt \\
&+\g\iint_Q f dx\,dt + \d^{-\frac1{p-1}}\g\left\|\partial_tv\right\|_{L^{p'}\big((t_1,t_2);\;W^{-1,p'}(B_R)\big)}.
\end{split}
\]
The second and third terms on the right hand side are estimated
by Corollary~\ref{lem-hardy-Kp}. Hence we have that
\[
\begin{split}
\sup_{\tau\in (t_1,t_2)}\int_{B_R}\big(u(\tau)-v(\tau)\big)^2 dx & +
\iint_Q |\n u-\n v|^pdx\,dt
\le \g \iint_Q \big|\n v\big|^pdx\,dt
 + \g\iint_Q f|v|^pdx\,dt \\
&+\g\iint_Q f dx\,dt + \g\left\|\partial_tv\right\|_{L^{p'}\big((t_1,t_2);\;W^{-1,p'}(B_R)\big)}.
\end{split}
\]
Finally, by the H\"{o}lder and Sobolev inequalities we conclude that
\[
\begin{split}
\iint_Q|u-v|^{p+\frac{2p}N}dx\,dt \le & \int\limits_{t_1}^{t_2}
\left(\,\int\limits_{B_R}|u-v|^{\frac{pN}{N-p}}dx\right)^{\frac{N-p}N}
\left(\,\int\limits_{B_R}|u-v|^2dx\right)^{\frac pN}dt
\\
\le & \g\iint_Q|\n u - \n v|^pdx\,dt
\left(\sup_{\tau\in (t_1,t_2)}\int_{B_R}\big(u(\tau)-v(\tau)\big)^2 dx\right)^{\frac pN}.
\end{split}
\]
\end{proof}

The preceding proposition together with Proposition~\ref{max_u}
turns the a priori estimate of Proposition~\ref{apriori1}
into an a posteriori one, as the following corollary states.

\begin{corollary}
\label{apriori1-1} Let conditions of
Proposition~\ref{funct_int_est} be fulfilled. Assume that
$(f^2+f_1^2+g_1^2+f_2^2+g_2^2)\ind_{B_R}\in PK$. Then, for every
$\a\ge 0$ and $l>0$ there exist $\beta>0$ and $\gamma_{l,\a}>0$
independent of $\eps$ such that
\[
\esssup_{t\in[t_1',t_2']}\int_{B'} |\n u_\eps|^{2+2\a}dx
+\iint_{Q'} \Big|D\big(\n u_\eps(|\n u_\eps|-l)_+^{\a+\frac p2-1}\big)\Big|^2dxd\tau\le \gamma_{l,\a},
\]
provided $(g_1^2+g_2^2)\ind_{B_R}\in PK_\beta$.
\end{corollary}

In the following proposition we prove that the solutions $u_\eps$ and certain functions of their gradients are locally
Lipschitz continuous in time variable, with values in certain Banach spaces, uniformly in $\eps$.  This will
be used to apply a compactness result from \cite{simon}.

\begin{proposition}\label{time_derivative}
Let assumptions of Corollary~\ref{apriori1-1} be fulfilled.
Then for all $\sigma>N$, $\a\ge p$ and $l>0$, $B'\Subset B_R$
and $(t_1',t_2')\Subset(t_1,t_2)$ there are constants
$\beta,\g>0$ independent of $\eps$ such that, if
$(g_1^2+g_2^2)\ind_{B_R}\in PK_\beta$, then, for all
$h\in(0,t_2'-t_1')$,
\[
\begin{split}
&\int\limits_{t_1'}^{t_2'-h}\Big\|u_\eps(t+h)-u_\eps(t)\Big\|_{W^{-1,\frac
p{p-1}}(B')}dt\le \g h;
\\
&\int\limits_{t_1'}^{t_2'-h}\Big\|\big(|\nabla u_\eps(t+h)|-l\big)_+^{\a}\nabla u_\eps(t+h)
-\big(|\nabla u_\eps(t)|-l\big)_+^{\a}\nabla u_\eps(t)\Big\|_{W^{-1,\frac
\sigma{\sigma-1}}(B')}dt\le \g h.
\end{split}
\]
\end{proposition}

\begin{proof}
The first assertion follows from Propositions~\ref{Dt}
and~\ref{funct_int_est}. 

To shorten the proof of the second assertion, we introduce some
notation. Let $F=g+g_1+g_2+f+f_1+f_2$. For $l>0$ and
$\alpha\ge2$ define $w_{l,\alpha}:\mathbb{R}^N\to\mathbb{R}^N$
as follows. For $\zeta\in\mathbb{R}^N$ let $w(\zeta)\equiv
w_{l,\alpha}(\zeta)=(|\zeta|-l)_+^\alpha\zeta$. Also, we define
$\mathfrak{h}_u:Q\to (\mathbb{R}^{N\times N})^*$ as follows.
For $M\in \mathbb{R}^{N\times N}$,
\[
\mathfrak{h}_u[M]:= \tr\big\{M(\partial_z\mathbf{A}\,D^2 u + \partial_x\mathbf{A})\big\}
+ \<M\partial_u\mathbf{A}, \nabla u\> + b(u,\nabla u)\tr M.
\]
Then it follows from Lemma~\ref{lem2} that, for any
$\tau\in(t_1,t_2-h)$
\[
\int \<\nabla u, \zeta\> dx\Big|_\tau^{\tau+h} - \int\limits_\tau^{\tau+h}\int\limits_B \<\nabla u, \partial_t\zeta\> dxdt
= \int\limits_\tau^{\tau+h}\int\limits_B
\mathfrak{h}_u[D\zeta]dxdt \quad \mbox{for all }\zeta\in H^1_c(Q\to\mathbb{R}^N ).
\]
Recall that we have to verify that, for all $B'\Subset B_R$,
$(t_1',t_2')\Subset (t_1,t_2)$ and $h\in(0,t_2'-t_1')$,
\[
\begin{split}
&\int\limits_{t_1'}^{t_2'-h}\left\|\<w(\n u(\tau+h))-w(\n u(\tau))\right\|_{W^{-1,\sigma'}(B')}dt \\
\equiv
&\int\limits_{t_1'}^{t_2'-h}\sup\limits_{\|\zeta\|_{W^{1,\sigma}_0(B')\le1}}
\left|\int\<w(\n u(\tau+h))-w(\n u(\tau)),\zeta\>dx \right|dt \le \g h,
\end{split}
\]
with some $\g>0$ independent of $h>0$ and $u$.

Note that, for a vector field $\zeta$ differentiable in $t$,
one has that $w(\zeta)$ is differentiable in $t$ and
\[
\partial_t w(\zeta) = (|\zeta|-l)_+^\alpha\partial_t\zeta + \alpha \frac{\zeta}{|\zeta|}(|\zeta|-l)_+^{\alpha-1}
\<\zeta,\partial_t\zeta\>.
\]
Hence by using the Steklov averaging one obtains that, for all
$\zeta\in C^1_c(B'\to\mathbb{R}^N )$,
\begin{equation}\label{w-equation}
   \int \<w(\nabla u), \zeta\> dx\Big|_\tau^{\tau+h}  =  \int\limits_\tau^{\tau+h}\int\limits_{B_R}
   \mathfrak{h}_u[D\tilde\zeta_1+D\tilde\zeta_2]dxdt,
\end{equation}
with
$
\tilde\zeta_1 = (|\nabla u|-l)_+^\alpha\zeta \mbox{ and }
\tilde\zeta_2 = \alpha \frac{\nabla u}{|\nabla u|}(|\nabla u|-l)_+^{\alpha-1}
\<\nabla u,\zeta\>.
$
So, for $B'\Subset B_R$, $(t_1',t_2')\Subset (t_1,t_2)$ and
$h\in(0,t_2-t_2')$,
 $\zeta\in C^1_c(B'\to\mathbb{R}^N )$, we have
\[
\int\limits_{t_1'}^{t_2'-h}\sup\limits_{\|\zeta\|_{W^{1,\sigma}_0(B')}\le1}\left|\int\<w(\n u(t+h))-w(\n u(t)),\zeta\>dx \right|dt \le
h\int\limits_{t_1'}^{t_2'}\sup\limits_{\|\zeta\|_{W^{1,\sigma}_0(B')}\le1}\int\limits_{B'} \big|\mathfrak{h}_u[D\tilde\zeta_1+D\tilde\zeta_2]\big|dxdt.
\]
Now observe that, by assumptions \eqref{e1.2b} and
\eqref{e1}--\eqref{e3a}, for every $M\in \mathbb{R}^{N\times
N}$,
\begin{equation}\label{h_est}
\begin{split}
\big|\mathfrak{h}_u[M]\big|\le &
\g|M|\big\{(|\n u|^{p-2}+1)|D^2 u| + (g+g_2)|\n u|^{p-1} + g_1 |\n u|^{p-2} + f + f_1 + f_2\big\}
\\
\le & \g_{p,l}|M|\big\{[(|\n u|-l)_+^{p-2}+1]|D^2 u| + F[(|\n
u|-l)_+^{p-1}+1]\big\}.
\end{split}
\end{equation}
In  turn, we compute that
\[
D\tilde\zeta_1 = (|\nabla u|-l)_+^\alpha D\zeta
+ \alpha \frac{\nabla u}{|\nabla u|}(|\nabla u|-l)_+^{\alpha-1}\zeta\otimes\left(D^2u \frac{\n u}{|\n u|}\right)
\]
and
\[
\begin{split}
D\tilde\zeta_2 = & \alpha \frac{\nabla u}{|\nabla u|}(|\nabla u|-l)_+^{\alpha-1}
\frac{\n u}{|\n u|}\otimes\big\{(D\zeta)^\top\n u + D^2u\zeta \big\}
\\
& + \alpha \frac{\nabla u}{|\nabla u|}(|\nabla u|-l)_+^{\alpha-1}\<\n u,\zeta\>
\left(\frac{D^2u}{|\n u|} - \frac{\n u}{|\n u|}\otimes \left(\frac{D^2u}{|\n u|}\frac{\n u}{|\n u|}\right)\right)
\\
& + \alpha (\alpha -1)\frac{\nabla u}{|\nabla u|}(|\nabla u|-l)_+^{\alpha-2}\<\n u,\zeta\>
\frac{\n u}{|\n u|}\otimes \left(D^2u \frac{\n u}{|\n u|} \right).
\end{split}
\]
Hence
\begin{equation}\label{Dzeta}
\begin{split}
&|D\tilde\zeta_1 + D\tilde\zeta_2|\le  \g_{\alpha}\big(|D\zeta|\,|\n u|(|\nabla u|-l)_+^{\alpha-1}
+ |\zeta|\,|D^2u|\,|\n u|(|\nabla u|-l)_+^{\alpha-2} \big)
\\
\le & \g_{\alpha,l}\left\{|D\zeta|\left[(|\nabla u|-l)_+^{\alpha} + (|\nabla u|-l)_+^{\alpha-1}\right]
+|\zeta|\,|D^2u|\left[(|\nabla u|-l)_+^{\alpha-1} + (|\nabla u|-l)_+^{\alpha-2}\right]\right\}.
\end{split}
\end{equation}
So it follows from \eqref{w-equation}--\eqref{Dzeta} that
\[
\begin{split}
\frac1{\g_{\alpha,p,l}}\int\limits_{B'} \big|\mathfrak{h}_u[D\tilde\zeta_1+D\tilde\zeta_2]\big|dx
\le &
\int\limits_{B'} |\zeta|\,|D^2u|^2\left[(|\nabla u|-l)_+^{p+\alpha-3} + (|\nabla u|-l)_+^{\alpha-2}\right]dx
\\
& + \int\limits_{B'} |\zeta|\,|D^2u|F\left[(|\nabla u|-l)_+^{p+\alpha-2} + (|\nabla u|-l)_+^{\alpha-2}\right]dx
\\
& + \int\limits_{B'} |D\zeta|\,|D^2u|\left[(|\nabla u|-l)_+^{p+\alpha-2} + (|\nabla u|-l)_+^{\alpha-1}\right]dx
\\
& + \int\limits_{B'} |D\zeta|F\left[(|\nabla u|-l)_+^{p+\alpha-1} + (|\nabla u|-l)_+^{\alpha-1}\right]dx.
\end{split}
\]
Now it follows from the H\"{o}lder inequality that
\[
\begin{split}
\frac1{\g_{\alpha,p,l}}\int\limits_{B'} \big|\mathfrak{h}_u[D\tilde\zeta_1+D\tilde\zeta_2]\big|dx
\le  \|\zeta\|_\infty &\int\limits_{B'} |D^2u|^2\left[(|\nabla u|-l)_+^{p+\alpha-3} + (|\nabla u|-l)_+^{\alpha-2}\right]dx
\\
+ \|\zeta\|_\infty & \left(\int\limits_{B'} F^2dx\right)^{\frac12}
\left(\int\limits_{B'} |D^2u|^2\left[(|\nabla u|-l)_+^{2p+2\alpha-4} + (|\nabla u|-l)_+^{2\alpha-4}\right]dx\right)^{\frac12}
\\
 +
\|D\zeta\|_2 &
\left(\int\limits_{B'}|D^2u|^2\left[(|\nabla u|-l)_+^{2p+2\alpha-4} + (|\nabla u|-l)_+^{2\alpha-2}\right]dx
\right)^{\frac12}
\\
 + \|D\zeta\|_N & \left(\int\limits_{B'}F^2dx\right)^{\frac12}
\left[\int\limits_{B'}(|\nabla u|-l)_+^{(p+\alpha-1)\frac{2N}{N-2}}
+\int\limits_{B'}(|\nabla u|-l)_+^{(\alpha-1)\frac{2N}{N-2}}
\right]^{\frac{N-2}{2N}}.
\end{split}
\]

Thus,
\[
\int\limits_{B'} \big|\mathfrak{h}_u[D\tilde\zeta_1+D\tilde\zeta_2]\big|dx
\le c(u)(\|\zeta\|_\infty + \|D\zeta\|_2+\|D\zeta\|_N)
\]
with
\[\begin{split}
c(u)=\g_{\a,p,l}\Big\{ &
\int\limits_{B'}|D^2u|^2\left[(|\nabla u|-l)_+^{2p+2\alpha-4} + (|\nabla u|-l)_+^{\alpha-2}\right]dx
\\
& + \big(\int\limits_{B'}\left[|\nabla u|^{p+\alpha-1} +|\nabla u|^{\alpha-1}\right]^{\frac{2N}{N-2}}
dx\big)^{\frac{N-2}N} +
\int\limits_{B'}F^2dx \Big\}.
\end{split}\]
Finally, by the Sobolev embedding theorem, for any $\sigma>N$,
one has $\|\zeta\|_\infty + \|D\zeta\|_2+ \|D\zeta\|_N\le c
\|\zeta\|_{W^{1,\sigma}_0(B')}$, and by
Corollary~\ref{apriori1-1}
$c(u)$ is bounded by a
constant independent of $u$ provided $\a\ge p$. So the second
assertion follows.
\end{proof}

 The following lemma serves to assert the pointwise convergence of the gradient.

\begin{lemma}\label{pointwise}
Let $\xi_n$ be a sequence of a.e. finite vector fields such
that there exists $\a>0$ such that
$\xi_n(|\xi_n|-\frac1m)_+^\a$ converges a.e. as $n\to \infty$
for all $m\in \mathbb{N}$. Then $\xi_n$ converges a.e. as $n\to
\infty$.
\end{lemma}

\begin{proof}
Denote $\eta_n:=\xi_n(|\xi_n|-\frac1m)_+^\a$ and
$E_{nm}:=\{x:|\xi_n|\ge\frac1m\}$.

Note that the function $\phi_m(s):=s(s-\frac1m)_+^\a$ is a
homeomorphism $[\frac1m,\infty)\to [0,\infty)$. Let $\psi_m$
denote the inverse map. Then one has
\[
\xi_n\chi_{E_{nm}}=\eta_n\frac{\psi_m(|\eta_n|)}{|\eta_n|}.
\]
So there are vector fields $\zeta_m$, $m\in\mathbb{N}$ such
that $\xi_n\chi_{E_{nm}}\to \zeta_m$ a.e. as $n\to\infty$.

Let
\[
E_m=\liminf\limits_{n\to\infty}E_{nm}=\bigcup\limits_{N\in\mathbb{N}}
\bigcap\limits_{n\geq N}E_{nm}=\{x:\liminf\limits_{n\to\infty}|\xi_n|(x)\ge\tfrac1m\},
~E:=\bigcup\limits_{m\in\mathbb{N}}E_m=\{x:\liminf\limits_{n\to\infty}|\xi_n|(x)>0\}.
\]
Then, for every $x\in E_m$ there exists $N\in\mathbb{N}$ such
that $x\in E_{nm}$ for all $n\ge N$. Hence
\[
\lim\limits_{n\to\infty}\xi_n(x) = \lim\limits_{n\to\infty}\xi_n(x)\chi_{E_{nm}}=\zeta_m(x) \mbox{ for all } x\in E_m .
\]
Thus $\xi_n(x)\to \xi(x)$ as $n\to\infty$ for all $x\in E$. Note
that $|\xi|(x)\ge\frac1m$ for all $x\in E_m$. So
$\zeta_m(x)=\xi\chi_{\{|\xi|\ge\frac1m\}}(x)$ for all $m\in
\mathbb{N}$ and $x\in E$.

Further,
\[
E^c:=
\{x: \forall m,N\in \mathbb{N}~ \exists n \ge N \mbox{ such that }|\xi_n|(x)<\tfrac1m\}=\{x:\liminf\limits_{n\to\infty}|\xi_n|(x)=0\}.
\]
Therefore, for all $x\in E^c$ and $m\in\mathbb{N}$,
\[
|\zeta_m|(x)=\liminf\limits_{n\to\infty}|\xi_{n}|\chi_{E_{nm}}(x)=0.
\]
Now we define $\xi(x)=0$ for $x\in E^c$ so that
$\zeta_m(x)=\xi\chi_{\{|\xi|\ge\frac1m\}}(x)$ a.e. Finally, we have
\[
\limsup\limits_{n\to\infty}|\xi_n-\xi|\le \limsup\limits_{n\to\infty}|\xi_n\chi_{E_{nm}}-\zeta_m|
+ \limsup\limits_{n\to\infty} |\xi_n|\chi_{\{|\xi_n|<\frac1m\}} + |\xi|\chi_{\{|\xi|<\frac1m\}}
<\tfrac2m\to0 \mbox{ as }m\to\infty.
\]
\end{proof}

\begin{proof}[Proof of Theorem \ref{PK-existence}]
For $\eps>0$, let $\A_\eps$, $b_\eps$, and $u_\eps$ be as in
Proposition~\ref{approx}. 
Let $Q':=(t_1',t_2')\times B' \Subset Q$. Due to the embedding
$W^{1,q}(B')\Subset L^1(B')\subset
W^{-1,\frac\sigma{\sigma-1}}(B')$ for any $q\ge1$ and
$\sigma>N$,
it follows from~\cite[Theorem 5]{simon} and
Corollary~\ref{apriori1-1} and Proposition~\ref{time_derivative}
that, for any $\a\ge p$, $l>0$, the sets $\{u_\eps\}_{\eps>0}$
and $\big\{(|\nabla u_\eps|-l)_+^{\a}\nabla
u_\eps\big\}_{\eps>0}$ are compact in $L^1(Q')$. Using a
compact exhaustion of $Q$ and a standard diagonalization  we
conclude that there exists a subsequence $\eps_n\downarrow0$
such that, $u_n=u_{\eps_n}$ converges as $n\to\infty$ a.e. on
$Q$ and $\n u_n(|\n u_n|-\tfrac1m)_+^{\a}$ converges as
$n\to\infty$ a.e. on $Q$ for all $m\in \mathbb{N}$. Then by
Lemma~\ref{pointwise} it follows that $\n u_n$ converges as
$n\to\infty$ a.e. on $Q$. Let $u$ denote the pointwise limits
of $u_n$. Since $\n u_n$ is uniformly bounded in $L^q_{loc}(Q)$
for all $q>1$, we conclude that  $\n u\in L^q_{loc}(Q)$ for all
$q>1$ and $\n u_n\to \n u$ as $n\to\infty$ weakly in
$L^q_{loc}(Q)$. Since the weak and the pointwise limits
coincide, $\n u_n\to \n u$ as $n\to\infty$ a.e. on $Q$.

Now observe that
\[
\begin{split}
|\A_{\eps_n}(u_n,\n u_n)- \A(u,\n u)|\le |\A(u_n,\n u_n)- \A(u,\n u)| + |u_n-u|\int_0^1 |\partial_u\A(\o_s)|ds\\
+|\n u_n-\n u|\int_0^1 |\partial_z\A(\o_s)|ds, \quad \text{where}\ \ \o_s=((1-s)u_n+su,(1-s)\n u_n+s\n u).
\end{split}
\]
Using the structure conditions \eqref{e2}, \eqref{e3} we
infer that
$\A_{\eps_n}(u_n,\n u_n)\to \A(u,\n u)$ as
$n\to\infty$ a.e. on $Q$ and that, due to \eqref{e1.2b}, the
set $\{\A_{\eps}(u_\eps,\n u_\eps)\}_{\eps>0}$ is bounded in
$L^{\frac p{p-1}}_{loc}(Q)$.
Hence $\A_{\eps_n}(u_n,\n u_n)\to \A(u,\n u)$ as $n\to\infty$
weakly in $L^{\frac p{p-1}}_{loc}(Q)$.
Now we note that
\[
\begin{split}
&|b_{\eps_n}(u_n,\n u_n)-  b(u,\n u)|\le |b(u_n,\n u_n) - b(u,\n u)|\\
&+
(\ind_{\{|b(u_n,\n u_n) - b(u,\n u)|\ge 1/2\}}+\ind_{\{|b(u,\n u)|>1/(2\eps_n)\}})(|b(u_n,\n u_n)| +|b(u,\n u)|).
\end{split}
\]
Hence, due to \eqref{e1.2b}  $b_{\eps_n}(u_n,\n u_n) \to b(u,\n u)$ a.e. on $Q$.
Then 
by \eqref{e1.2b},
the set $\{b_{\eps_n}(u_n,\n u_n)\}$ is bounded in
$L^{\frac p{p-1}}_{loc}(Q)$. 
So $b_{\eps_n}(u_n,\n u_n)$ is
weakly compact in $L_{loc}^{p'}(Q)$ and hence $b(u,\n u)\in
L_{loc}^{p'}(Q)$ and $b_{\eps_n}(u_n,\n u_n) \to b(u,\n u)$ as $n\to
\infty$ weakly in $L_{loc}^{p'}(Q)$. Hence, for every $\theta\in
W_c^{1,p}(B)$, we have that
$
\iint\limits_Q \A_{\eps_n}(u_n,\n u_n)\n \theta dx\,d\tau \to \iint\limits_Q \A(u,\n u)\n \theta dx\,d\tau$ and
$\iint\limits_Q b_{\eps_n}(u_n,\n u_n)\theta dx\,d\tau \to \iint\limits_Q b(u,\n u)\theta dx\,d\tau$
as $n\to\infty$. Thus $u$ is a solution to \eqref{e0} satisfying estimate \eqref{main-est}.
\end{proof}


\section{Proof of Theorem~\ref{PK-aposteriori}}

In the proof we follow the idea from~\cite{tolksdorf}, with required modifications.
We start with the following technical lemma.

\begin{lemma}\label{DA}
There exist $c_p,\Gamma_p>0$ such that, for all $(x,t)\in
\Omega_T$, $\mu, \tilde\mu\in \mathbb{R}$, $\eta, \tilde\eta\in
\mathbb{R}^N$, one has
\begin{equation}\label{DA-est}
\begin{split}
&\<\A(x,t,\mu,\eta)   - \A(x,t,\tilde\mu,\tilde\eta),\eta - \tilde\eta\>
\ge  c_p(|\eta| + |\tilde\eta|)^{p-2}|\eta - \tilde\eta|^2\\
& - \Gamma_p\Big(f_1^{p'}(x)|\mu - \tilde\mu|^{p'} + g_1^2(x)|\tilde\eta|^{p-2}|\mu - \tilde\mu|^2
+g_1^2(x)|\eta-\tilde\eta|^{p-2}|\mu - \tilde\mu|^2\Big).
\end{split}
\end{equation}
\end{lemma}

\proof
Set $\omega_s:=(x,t,s\mu + (1-s)\tilde\mu,s\eta
+(1-s)\tilde\eta)$, $s\in [0,1]$. Then
\[
\A(x,t,\mu,\eta) - \A(x,t,\tilde\mu,\tilde\eta) = \int\limits_0^1\partial_z\A(\omega_s)(\eta - \tilde\eta)ds
+ \int\limits_0^1\partial_u\A(\omega_s)(\mu - \tilde\mu)ds.
\]
Then, by \eqref{e1}, there exist $c_{0,p}>0$ such that
\[
\int\limits_0^1\<\partial_z\A(\omega_s)(\eta - \tilde\eta), \eta - \tilde\eta\>ds
\ge c_{0,p}(|\eta| + |\tilde\eta|)^{p-2}|\eta - \tilde\eta|^2.
\]
Further, by \eqref{e3}, there exists $C_p$ such that
\[
\begin{split}
\Big|\int\limits_0^1\<\partial_u\A(\omega_s)(\mu - \tilde\mu), \eta - \tilde\eta\>ds\Big|
\le & f_1(x)|\mu - \tilde\mu|\,|\eta - \tilde\eta|
+ C_pg_1(x)(|\eta| + |\tilde\eta|)^{p-2}|\mu - \tilde\mu|\,|\eta - \tilde\eta|
\\
\le & \frac{c_{0,p}}4|\eta - \tilde\eta|^{p} + \frac1{c_{0,p}}f_1^{p'}(x)|\mu - \tilde\mu|^{p'}
+ \frac{c_{0,p}}4 (|\eta| + |\tilde\eta|)^{p-2}|\eta - \tilde\eta|^2
\\
& + \frac{C^2_p}{c_{0,p}}g_1^2(x)(|\eta| + |\tilde\eta|)^{p-2}|\mu - \tilde\mu|^2
\\
\le & \frac{c_{0,p}}2 (|\eta| + |\tilde\eta|)^{p-2}|\eta - \tilde\eta|^2
+ \frac1{c_{0,p}}f_1^{p'}(x)|\mu - \tilde\mu|^{p'}
\\
& + \frac{2^{p-2}C^2_p}{c_{0,p}}g_1^2(x)|\eta-\tilde\eta|^{p-2}|\mu - \tilde\mu|^2
+\frac{4^{p-2}C^2_p}{c_{0,p}}g_1^2(x)|\tilde\eta|^{p-2}|\mu - \tilde\mu|^2. \qed
\end{split}
\]

Similar to what was done in \cite{tolksdorf} we introduce the following functions:
\[
\begin{split}
\widehat b(x,t,\tilde\mu,\tilde\eta) & :=
\Gamma_p\Big(f_1^\frac{p}{p-1}(x)|u(x,t) - \tilde\mu|^\frac{2-p}{p-1} + g_1^2(x)|\tilde\eta|^{p-2}\Big)(u(x,t)-\tilde\mu),
\\
\ol{b}(x,t,\tilde\mu,\tilde\eta)& :=
\big(-f(x)(1+|u(x,t)|^{p-1})-g(x)|2\tilde\eta|^{p-1}\big)\vee b\big(x,t,u(x,t), \n u(x,t)\big)\wedge\\
&\wedge \big(f(x)(1+|u(x,t)|^{p-1})+g(x)|2\tilde\eta|^{p-1}\big).
\end{split}
\]
Set
\[
\widetilde b(x,t,\tilde\mu,\tilde\eta)=\widehat b(x,\tilde\mu,\tilde\eta)+\ol{b}(x,\tilde\mu,\tilde\eta).
\]
Consider the auxiliary the equation
\begin{equation}\label{aux-eq}
    \partial_t\tilde u - \mathrm{div}\A(\tilde u, \n \tilde u) = \widetilde b(\tilde u, \n \tilde u).
\end{equation}

\begin{proposition}\label{coincide}
Let $Q=B_R\times(t_1,t_2)\Subset \O_T$.
Let $\tilde u$ be a weak solution to \eqref{aux-eq} in $Q$ such that
\[
u|_{\P Q}=\tilde u|_{\P Q},
\]
where $\P Q$ is the parabolic boundary of $Q$.
Then
$\tilde u=u$ in $Q$ if $\sup\limits_{x\in B_R}
W_p^{g^p}(x,2R)$ and $\sup\limits_{x\in B_R}
W_p^{g_1^p}(x,2R)$ are small enough.
\end{proposition}

\begin{proof}
Subtract \eqref{aux-eq} out
of \eqref{e0} and multiply the difference by $u-\tilde u$. Note
that the latter belongs to $L^p\Big((t_1,t_2)\to
W^{1,p}_0(B)\Big)\cap C_0\Big([t_1,t_2)\to L^2(B)\Big)$.  We
obtain that
\[
\begin{split}
\frac12\int\limits_{B}|u-\tilde u|^2(t_2)dx
+ \iint\limits_Q \<\A(u,\n u)  - \A(\tilde u,\n\tilde u),\n u - \n\tilde u\>dxdt
= & \iint\limits_Q \big(b(u,\n u) - \widetilde b(\tilde u,\n\tilde u)\big)(u-\tilde u)dxdt.
\end{split}
\]
By Lemma \ref{DA} we have
\[
\begin{split}
\iint\limits_Q \<\A(u,\n u)  - \A(\tilde u,\n\tilde u),\n u - \n\tilde u\>dxdt
+ \iint\limits_Q \widehat b(\tilde u,\n\tilde u)(u-\tilde u)dxdt\\
\ge  c_p\iint\limits_Q (|\n u| + |\n\tilde u|)^{p-2}|\n u - \n\tilde u|^2dxdt
 - \Gamma_p \|\n u-\n\tilde u\|_{p}^{p-2}\|g_1(u - \tilde u)\|_p^2.
\end{split}
\]
Further, note that $b(u,\n u)$ is of the same sign that
$\ol{b}(\tilde u,\n\tilde u)$. Also observe that $b(u,\n u)\not=
\ol{b}(\tilde u,\n\tilde u)$
only under the condition $|b(u,\n u)|>f(1+|u|^{p-1}) +
g|2\n\tilde u|^{p-1}$, which implies that $|\n u|\ge 2|\n\tilde
u|$. Hence
\[
|b(u,\n u) - \ol{b}(\tilde u,\n\tilde u)|\le g|\n u|^{p-1}\ind_{\{|\n u|\ge 2|\n\tilde u|\}}
\le 2^{p-1}g|\n u - \n\tilde u|^{p-1}.
\]
Therefore
\[
\iint\limits_Q|b(u,\n u) - \ol{b}(\tilde u,\n\tilde u)|\,|u-\tilde u|dxdt
\le 2^{p-1}\|\n u-\n\tilde u\|_{p}^{p-1}\|g(u - \tilde u)\|_p^p.
\]
Thus we obtain that
\[
c_p\iint\limits_Q (|\n u| + |\n\tilde u|)^{p-2}|\n u - \n\tilde u|^2dxdt
\le \Gamma_p \|\n u-\n\tilde u\|_{p}^{p-2}\|g_1(u - \tilde u)\|_p^2
+ 2^{p-1}\|\n u-\n\tilde u\|_{p}^{p-1}\|g(u - \tilde u)\|_p^p.
\]
By \eqref{hardy-Kp} this implies that
\[
c_p\|\n u - \n\tilde u\|_p^p \le \left\{\Gamma_p\sup\limits_{B_R} \left(W_p^{g_1^p}(x,2R)\right)^{\frac2{p'}}
+ 2^{p-1}\sup\limits_{B_R}\left(W_p^{g^p}(x,2R)\right)^{\frac1{p'}}\right\}\|\n u - \n\tilde u\|_p^p.
\]
So if $\sup\limits_{B_R}W_p^{g^p}(x,2R)$ and
$\sup\limits_{B_R}W_p^{g_1^p}(x,2R)$ are small enough then
$\|\n u - \n\tilde u\|_p^p \le 0$.
\end{proof}
\begin{proof}[Proof of Theorem~\ref{PK-aposteriori}]
Note that $|\ol{b}(x,t,\tilde\mu,\tilde\eta)|\le
f(x)+g(x)|2\tilde\eta|^{p-1}$ and
\[
|\widehat b(x,t,\tilde\mu,\tilde\eta)|\le
\Gamma_p\Big(g_1(x)|\tilde\eta|^{p-1} +  (f_1(x)^\frac{p}{p-1}+g_1(x)^p)|u(x,t) - \tilde\mu|^{p-1}
+ f_1(x)^\frac{p}{p-1}\Big).
\]
Hence equation \eqref{aux-eq} satisfies the structural
conditions \eqref{e1.2b}-\eqref{e3a} with
$2^{p-1}\Gamma_p(f_1^\frac{p}{p-1}+g_1^p + \sup |u|) + f$ and
$(2^{p-1}g+ \Gamma_pg_1)$ replacing $f$ and $g$, respectively.
Therefore by Theorem~\ref{PK-existence}, there exists a solution
$\tilde u$ coinciding with $u$ on the parabolic boundary of
$Q$, which enjoys the estimate \eqref{main-est}. Since $g^p\in
K_p$, we can choose $R$ so small that $u=\tilde u$ on $Q$.
Hence the assertion follows.


\end{proof}

\section{Boundedness of the gradient. Proof of Theorem~\ref{main}}
We obtain uniform estimates of the gradients on the sets where
$|\n u|>l$ for some positive $l$.  This restriction allows us
to simplify the structure conditions putting
$F=f+g+f_1+g_1+f_2+g_2$ and requiring
\begin{equation}
\label{e3new}
|\part_x\A|+|\part_u\A||z|+|b|\le F(x) |z|^{p-1}
\end{equation}
instead of the last condition in \eqref{e1.2b} and \eqref{e3}, \eqref{e3a}.
In obtaining the estimates we follow the parabolic version of the Kilpel\"ainen--Mal\'y technique \cite{KiMa, MZ} (see~\cite{LS2,Skr1}).

Let $\l>0$, $\d>0$, $l\ge l_0\ge 1$,
$\displaystyle
\sigma :=\left(\frac{|\n u|^2-l}\d\right)_+.
$
Set
\[
\var(\sigma):=\int_0^{\sigma}(1+s)^{-1-\l}ds, \quad G(\sigma)=\int_0^\sigma s\var(s)ds.
\]


Before formulating the next lemma let us note that
$G(\sigma)\asymp \min\{\sigma^2,\sigma^3\}, \sigma>0$ and
\begin{equation}
\label{e7}
\var(\sigma)\asymp \frac{\sigma}{\sigma+1}\ge \frac{\sigma}{(1+\sigma)^{1+\l}}  \quad\mbox{ so that }\
\part_\sigma(\sigma\var(\sigma))\asymp \var(\sigma).
\end{equation}

\begin{lemma}
\label{lem4} Let $\d>0$. With notation $w=|\n u|^2$, $\sigma=\left(\frac{w-l}{\d}\right)_+$ the following inequality holds
\begin{eqnarray}
\nonumber
\esssup_t \int G(\sigma(t)) \xi^q(t) +\iint w^{\frac p2-1} |\n\sigma|^2 \var(\sigma) \xi^q
\le \g \iint \sigma^2 \xi^{q-1}|\part_t\xi| + \g \iint w^{\frac p2-1} |\n \xi|^2 \sigma^2 \var(\sigma)\xi^{q-2} \\
\label{e-lem4}
+ \g \d^{-2} l^{\frac p2+1} \iint F^2 \xi^q
+\g \d^{\frac p2-1} \iint F^2 \sigma^{\frac p2+1} \var(\sigma) \xi^q.
\end{eqnarray}
\end{lemma}
\begin{proof}
Due to the last assertion of Theorem~\ref{PK-existence}, we may
assume that $u\in L^2_{loc}\big((0,T);\;W^{2,2}_{loc}(\O)\big)\cap
C\big((0,T);\;W^{1,2}_{loc}(\O)\big)$. By Lemmas~\ref{lem2} and
~\ref{lem3} with $\zeta= (\n u) \sigma\var(\sigma)\xi^q$ we have
\begin{equation}
\label{e6}
\frac\d2\int\limits_\O G(\sigma(t)) \xi^q(t) +\iint\limits_{(0,t)\times\O} \TR{(D\zeta)(\part_z\A)D^2u }
\le q\frac\d2\iint\limits_{(0,t)\times\O} G(\sigma) \xi^{q-1}\part_\tau\xi + \iint\limits_{(0,t)\times\O} (|\part_x\A|+|\part_u \A||\n u| +|b|)|D\zeta|.
\end{equation}
Note that $ D\zeta= D^2 u \sigma \var(\sigma) \xi^q + (\n u
\otimes \n\sigma)(\var(\sigma)+\sigma(1+\sigma)^{-1-\l})\xi^q+
q (\n u \otimes \n\xi)\sigma \var(\sigma) \xi^{q-1}, $ and
$D^2u\n u=\frac12\n w =\d/2\n\sigma$. By structure conditions
\eqref{e1} and \eqref{e2}, Remark~\ref{rem1} and \eqref{e7} it follows that
\[
\TR{(D\zeta)(\part_z\A)D^2 u}\ge c_0 w^{\frac{p}2-1} |D^2 u|^2_{HS} \sigma \var(\sigma) \xi^q +
\frac{c_0}2\d w^{\frac{p}2-1}|\n \sigma|^2 \var(\sigma) \xi^q -\frac{c_1}2\d w^{\frac{p}2-1} |\n\sigma|\,|\n\xi|\sigma \var(\sigma)\xi^{q-1}.
\]
By the Schwartz inequality
\[
\TR{(D\zeta)(\part_z\A)D^2 u}\ge c_0 w^{\frac{p}2-1} |D^2 u|^2_{HS} \sigma \var(\sigma) \xi^q +
\frac{c_0}4\d w^{\frac{p}2-1}|\n \sigma|^2 \var(\sigma) \xi^q-\frac{c_1^2}{4c_0}\d w^{\frac{p}2-1}|\n\xi|^2\sigma^2 \var(\sigma)\xi^{q-2}.
\]
To estimate the right hand side of \eqref{e6} we note that
\begin{eqnarray}
\nonumber
(|\part_x\A|+|\part_u \A||\n u| +|b|)|D\zeta| \\
\nonumber
\le
F w^{\frac{p-1}2}|D^2 u|\sigma \var(\sigma) \xi^q + 2F w^{\frac p2} |\n\sigma|\var(\sigma) \xi^q + q F w^{\frac p2} |\n\xi|\sigma\var(\sigma) \xi^{q-1}\\
\nonumber
\le \frac{c_0}{16} w^{\frac{p}2-1}|D^2 u|^2 \sigma \var(\sigma) \xi^q
+ \frac{c_0}{16}\d w^{\frac{p}2-1}|\n \sigma|^2 \var(\sigma) \xi^q\\
\label{e9}
+\g \d w^{\frac{p}2-1}|\n\xi|^2\sigma^2\var(\sigma) \xi^{q-2}
+\frac{\g}{\d} F^2 w^{\frac{p}2+1}\var(\sigma) \xi^q,
\end{eqnarray}
where we used the following obvious inequality $\frac1{\d}w^{\frac
p2+1}\ge \sigma w^{\frac p2}$. Thus we have from \eqref{e6}
\begin{eqnarray}
\nonumber
\d\int\limits_\O G(\sigma(t))\xi^q(t) +\d \iint\limits_{(0,t)\times\O} w^{p/2-1} |\n\sigma|^2\var(\sigma) \xi^q\\
\label{e10}
\le \g \d \iint\limits_{(0,t)\times\O} G(\sigma) \xi^{q-1}\part_\tau\xi
+\g\d \iint\limits_{(0,t)\times\O} w^{\frac p2-1}|\n\xi|^2\sigma^2 \var(\sigma) \xi^{q-2}
+ \frac{\g}{\d}\iint\limits_{(0,t)\times\O} F^2  w^{\frac p2+1}\var(\sigma) \xi^q.
\end{eqnarray}
To complete the proof note that $\var(\sigma)\le \frac1{\l}$,
$G(\sigma) \le \g\sigma^2$ and $w^{\frac p2+1}\le
\g(l^{p/2+1}+\d^{p/2+1}\sigma^{p/2+1})$.
\end{proof}

The next lemma provides the estimate of the last term in the right hand side of \eqref{e-lem4}.
\begin{lemma}
\label{lem5}
Let $h\in H^1_0(B)\cap L^\infty(B)$ be such that $-\Delta h=F^2$. Then
\[
\d^{\frac p2-1}\iint F^2 \sigma^{\frac p2+1}\var(\sigma)\xi^q  \le \g \|h\|_\infty \iint |\n\sigma|^2 w^{\frac p2-1} \var(\sigma)\xi^q
+\g \|h\|_\infty \iint |\n\xi|^2 \sigma^2 w^{\frac p2-1} \var(\sigma)\xi^{q-1}.
\]
\end{lemma}
\begin{proof}
Recall that $\var(\sigma)\asymp \frac{\sigma}{\sigma+1}$ and
apply \eqref{hardy}. Then
\begin{eqnarray*}
\iint F^2 \sigma^{\frac p2+1}\var(\sigma)\xi^q \le \g \iint F^2\frac{\sigma^{\frac p2+2}}{\sigma +1}\xi^q\\
\le \g \|h\|_\infty \iint |\n\sigma|^2 \sigma^{\frac p2-1}\frac{\sigma}{\sigma+1}\xi^q
+\g \|h\|_\infty \iint |\n\xi|^2\sigma^{\frac p2+1}\xi^{q-2}.
\end{eqnarray*}
Finally, note that $\d^{\frac p2-1}\sigma^{\frac p2-1}\le
w^{\frac p2-1}$.
\end{proof}

\bigskip

Now let $(x_0,t_0)\in \O_T$. Given $r,\d,l>0$ we denote
$\Delta:=\max\{\d,l\}$ and
\[
Q=Q^{x_0,t_0}_{r,\Delta}=B_r(x_0)\times I_\Delta\equiv B_r(x_0)\times (t_0-\frac{r^2}{\Delta^{\frac{p}2-1}}, t_0+\frac{r^2}{\Delta^{\frac{p}2-1}}).
\]

$\xi\in C^1_c(B_1(0)\times(-1,1))$, $0\le\xi\le1$, $\xi=1$ on
$B_{\frac12}(0)\times (-\frac12,\frac12)$,
$\xi_{r,\Delta}(x,t):= \xi(x_0 + r^{-1}x, t_0 + \Delta^{\frac
p2-1}r^{-2}t)$.

Then $\xi_{r,\Delta}\in C_c^1(Q^{x_0,t_0}_{r,\Delta})$,
$0\le\xi_{r,\Delta}\le1$, $\xi_{r,\Delta}=1$ on $\frac12Q$, and
$|\n\xi_{r,\Delta}|\le {\g}r^{-1}$ and
$|\part_\tau\xi_{r,\Delta}|\le \g \Delta^{\frac{p}2-1}r^{-2}.$

Set
\[
\Phi(w)=\int_0^{w^+} s^{\frac p4}(1+s)^{-\frac12}ds \asymp \min\{w^{\frac p4 +1}, w^{\frac{p+2}4 }\}, \quad
\Psi(w)= \int_0^{w^+} s^{\frac12}(1+s)^{-\frac12}ds \asymp \min\{w^{\frac 32}, w\}.
\]

\begin{corollary}
\label{cor6} If $Q^{x_0,t_0}_{r,\Delta}\Subset \O_T$ then
\begin{eqnarray*}
\esssup_{t} \int G(\sigma(t)) \xi_{r,\Delta}^q(t) +\d^{\frac p2-1} \iint |\n\Phi(\sigma)|^2\xi_{r,\Delta}^q
+
l^{\frac p2-1} \iint |\n\Psi(\sigma)|^2\xi_{r,\Delta}^q
\\
\le
\g \Delta^{\frac p2-1} r^{-2} \iint \sigma^2 \xi_{r,\Delta}^{q-2}
+\g \d^{\frac p2-1} r^{-2}\iint \sigma^{\frac p2+1}\xi_{r,\Delta}^{q-2}
+\g \left(\frac{l}{\d}\right)^2 r^2 \int_{B_r(x_0)}F^2.
\end{eqnarray*}
\end{corollary}

\bigskip
Set $r_j=r_0 2^{-j}$,  $B_j=B_{r_j}(x_0)$,   $l_{j+1}=l_j+\d_j$, $l_0=1$,   $\Delta_j =\max\{l_j,\d_j\}$, $I_j=I_{\Delta_j}$, $Q_j=B_j\times I_j$,
$\xi_j=\xi_{r_j,\Delta_j}$.
With this notation the next lemma is easy to check.
\begin{lemma}\label{emb}
If $\d_j>\left(\frac12\right)^\frac2{p-2}\d_{j-1}$ then $I_j\subset \frac12 I_{j-1}$.
\end{lemma}

\bigskip
Let
$L_j=\{(x,t)\in Q_j\,:\,w(x,t)>l_j\}$,
$L_j(t)=\{x\in B_j \,:\,w(x,t)>l_j\}$.
Fix $\varkappa>0$ a small number which will be chosen later depending on the known data.

Define

\begin{eqnarray}
\nonumber
A_j(l)=\sup_{t\in I_j} \frac1{r_j^N}\int_{L_j(t)} G\left(\frac{w-l_j}{l-l_j}\right)\xi_j^q dx
+ \frac{(l-l_j)^{\frac p2-1}}{r_j^{N+2}}\iint_{L_j}\left(\frac{w-l_j}{l-l_j}\right)^{\frac p2+1}\xi_j^{q-2}dx\,dt\\
\label{A}
+\frac{\Delta_j(l)^{\frac p2-1}}{r_j^{N+2}}\iint_{L_j}\left(\frac{w-l_j}{l-l_j}\right)^2\xi_j^{q-2}dx\,dt,
\end{eqnarray}
where $\Delta_j(l)=\max\{l_j,l-l_j\}$.

Set
\[
F_j=\left(\frac1{r_j^{N-2}}\int_{B_j}F^2(x)dx\right)^\frac12, \quad j=1,2,\dots.
\]

The sequence $(l_j)_{j\in\N}$ is defined inductively. We set as above $l_0=1$. Suppose $l_1,\dots,l_j$ have been defined. We show
how to define $l_{j+1}$.

First, note that $A_j(l)$ is continuous and $A_j(l)\to 0$ as
$l\to\infty$. If $A_j(l_j+F_j)\le \varkappa$ then we set
$l_{j+1}=l_j+F_j$. If on the other hand $A_j(l_j+F_j)>
\varkappa$ then there exists $\tilde l>l_j+F_j$ such that
$A_j(\tilde l)=\varkappa$, and we set $l_{j+1}=\tilde l$. In
both cases
\begin{equation}
\label{e35}
A_j(l_{j+1})\le \varkappa.
\end{equation}

\begin{lemma}
\label{lem3.1}
\begin{equation}
\label{e36}
\d_j\le \left(\frac12\right)^{\frac2{p-2}}\d_{j-1} +\g l_j F_j.
\end{equation}
\end{lemma}
\begin{proof}
Fix $j\ge 1$ and suppose that $\d_j>
\left(\frac12\right)^{\frac2{p-2}}\d_{j-1}$ and $\d_j>F_j$
since otherwise there is nothing to prove. This implies that
$A_j(l_{j+1})=\varkappa$.

We denote $\sigma_j:=\frac{w-l_j}{\d_j}$,
$\Phi_j:=\Phi(\sigma_j)$, $\Psi_j:=\Psi(\sigma_j)$.

{\it Claim}. $\sup\limits_{t\in I_j}\frac1{r_j^N}|L_j(t)|\le \g
\varkappa$.
Indeed,
for $(x,t)\in L_j$ one has
\begin{equation}
\label{Lj}
\frac{w(x,t)-l_{j-1}}{\d_{j-1}}=1+\frac{w(x,t)-l_{j}}{\d_{j-1}}\ge 1.
\end{equation}
Note that Lemma~\ref{emb} yields $\xi_{j-1}=1$ on $Q_j$. Hence
\begin{eqnarray*}
&& r_j^{-N}\sup_{t\in I_j}|L_j(t)|\le r_j^{-N}\sup_{t\in I_j} \int_{L_j(t)}  G\left(\frac{w-l_{j-1}}{\d_{j-1}}\right)\xi_{j-1}^{q}dx\\
&&\le 2^Nr_{j-1}^{-N}\sup_{t\in I_{j-1}} \int_{L_{j-1}(t)}  G\left(\frac{w-l_{j-1}}{\d_{j-1}}\right)\xi_{j-1}^{q}dx \le 2^N\varkappa,
\end{eqnarray*}
which proves the claim.

Now decompose $L_j$ as
$L_j=L_j\pr\cup L_j^{\prime\prime}$,
\begin{equation}
\label{decomp}
L_j\pr=\left\{(x,t)\in L_j\,:\,\frac{w(x,t)-l_j}{\d_j}<\eps \right\}, \quad L_j^{\prime\prime}=L_j\setminus L\pr_j,
\end{equation}
where $\eps$ depending on the data is small enough to be determined later. Then
\begin{equation}
\label{e37}
\frac{\d_j^{\frac p2-1}}{r_j^{N+2}}\iint_{L'_j}\sigma_j^{\frac p2+1}\xi_j^{q-2}dx\,dt
+\frac{\Delta_j^{\frac p2-1}}{r_j^{N+2}}\iint_{L'_j}\sigma_j^2\xi_j^{q-2}dx\,dt
\le \g\eps^2(1+\eps^{p/2-1})\sup\limits_{t\in I_j}\frac1{r_j^N}|L_j(t)|\le \g\eps^2\varkappa.
\end{equation}

Now recall that $\Phi(\sigma)\asymp \min\{\sigma^{\frac
p4+1},\sigma^{\frac{p+2}4}\}$. So $\sigma_j^{\frac p2+1}\le
\g(\eps)\Phi_j^2$ on $L^{\prime\prime}_j$. So we have
\begin{equation}
\label{e38}
\begin{split}
\frac{\d_j^{\frac p2-1}}
{r_j^{N+2}}
\iint_{L^{\prime\prime}_j}\sigma_j^{\frac
p2+1}
\xi_j^{q-2}dx\,dt
\le & \g(\eps)\frac{\d_j^{\frac p2-1}}
{r_j^{N+2}}
\iint_{L^{\prime\prime}_j} \left(\Phi_j\xi_j^{\frac q2-1}\right)^2dx\,dt
\\
\le & \g(\eps)\frac{\d_j^{\frac p2-1}}
{r_j^{N+2}} \int\limits_{I_j}|L_j(t)|^{\frac2N}
\left(\int_{L_j(t)}\left(\Phi_j\xi_j^{\frac
q2-1}\right)^{\frac{2N}{N-2}}dx\right)^{\frac{N-2}N}dt
\\
\le & \g(\eps)\frac{\d_j^{\frac p2-1}}
{r_j^{N+2}}\left(\sup\limits_{t\in I_j}|L_j(t)|
\right)^{\frac2N}\iint_{L_j}\left|\n(\Phi_j\xi_j^{\frac
q2-1})\right|^2dx\,dt
\\
\le & \g(\eps) \left(\sup\limits_{t\in I_j}\frac1{r_j^N}|L_j(t)|
\right)^{\frac2N}\frac{\d_j^{\frac p2-1}}
{r_j^{N}}
\iint \left|\n(\Phi_j\xi_j^{\frac
q2-1})\right|^2dx\,dt\\
\le & \g(\eps)
\varkappa
^{\frac2N}\frac{\d_j^{\frac p2-1}}
{r_j^{N}}
\iint \left|\n(\Phi_j\xi_j^{\frac
q2-1})\right|^2dx\,dt.
\end{split}
\end{equation}
Similarly, if $l_j\ge \d_j$,
\begin{equation}
\label{e39}
\begin{split}
\frac{\Delta_j^{\frac p2-1}}
{r_j^{N+2}}
\iint_{L^{\prime\prime}_j}\sigma_j^{2}
\xi_j^{q-2}dx\,dt
\le \g(\eps) \left(\sup\limits_{t\in I_j}\frac1{r_j^N}|L_j(t)|
\right)^{\frac2N}\frac{l_j^{\frac p2-1}}
{r_j^{N}}
\iint \left|\n(\Psi_j\xi_j^{\frac
q2-1})\right|^2dx\,dt\\
\le \g(\eps)
\varkappa^{\frac2N}
\frac{l_j^{\frac p2-1}}
{r_j^{N}}
\iint \left|\n(\Psi_j\xi_j^{\frac
q2-1})\right|^2dx\,dt.
\end{split}
\end{equation}

Using Corollary~\ref{cor6} we have
\begin{equation}
\label{e40}
\begin{split}
\frac{\d_j^{\frac p2-1}}{r_j^{N+2}}\iint_{L''_j}\sigma_j^{\frac p2+1}\xi_j^{q-2}dx\,dt
+\frac{\Delta_j^{\frac p2-1}}{r_j^{N+2}}\iint_{L''_j}\sigma_j^2\xi_j^{q-2}dx\,dt\\
\le \g(\eps) \varkappa^\frac2{N}
\left[ \varkappa +\left(\frac{l_j}{\d_j}\right)^2 r_j^{2-N}\int_{B_j} F^2 dx\right].
\end{split}
\end{equation}

Now we estimate the first term in the right hand side of \eqref{A} using Corollary~\eqref{cor6} and the Claim.
\begin{equation}
\label{e41}
\begin{split}
\sup_{t\in I_j} \frac1{r_j^N}\int_{L_j(t)} G\left(\frac{w-l_j}{l-l_j}\right)\xi_j^q dx\\
\le \g \eps^2(1+\eps^{p/2-1})\varkappa
+\left(\frac{l_j}{\d_j}\right)^2F_j^2
+ \g(\eps) \varkappa^\frac2{N} \left[ \varkappa
+\left(\frac{l_j}{\d_j}\right)^2 r_j^{2-N}\int_{B_j} F^2
dx\right].
\end{split}
\end{equation}
Collecting \eqref{e37}--\eqref{e41} we obtain
\[
\varkappa\le \g \eps^2(1+\eps^{p/2-1})\varkappa +\left(\frac{l_j}{\d_j}\right)^2F_j^2+
\g(\eps) \varkappa^\frac2{N}\left[ \varkappa +\left(\frac{l_j}{\d_j}\right)^2F_j^2\right].
\]
Now first choosing $\eps$ by the condition
\[
\g \eps^2(1+\eps^{p/2-1})=\frac14,
\]
and then $\varkappa$ such that
\[
\g(\eps)\varkappa^\frac2{N}=\frac14,
\]
we arrive at \eqref{e36}.
\end{proof}

\bigskip
Summing up the inequalities \eqref{e36} with respect to $j$ from 1 to $J-1$ we obtain
\[
l_J\le \g \d_0 + \g l_J \sum_{j=1}^{J-1} F_j.
\]
Choosing $r_0$ small enough so that $\int_0^{r_0}\frac{dr}r\left(\frac1{r^{N-2}}\int_{B_r(x_0)}F^2(y)dy\right)^\frac12<\frac1{2\g}$
we arrive at
\begin{equation}
\label{lJ}
l_J\le \g \d_0.
\end{equation}
It remains to estimate $\d_0$. From \eqref{A} we have
\[
\d_0\le \left(\frac1{r_0^N}\sup_t\int_{B_0}|\n u|^4\xi_0^qdx\right)^\frac12
+\left(\frac1{r_0^{N+p}}\iint_{Q_0} |\n u|^{p+2}\xi_0^{q-2}dx\,d\tau\right)^\frac12.
\]
By the iteration argument of Proposition~\ref{apriori1} with
$\a=l=1$, we obtain, with $Q_0\Subset Q\Subset \O_T$,
\begin{equation}
\label{d0}
\d_0\le \g(r_0^{-\frac N2} + r_0^{-\frac{N+p}2})
\left(\iint_{Q}|\n u|^pdxdt + \iint_Q(F^2+1)dxdt + \left(\iint_Q(F^2+1)dxdt\right)^{\frac N{2(N+2)}}\right).
\end{equation}
It follows from \eqref{lJ} that the sequence $(l_j)_j$ converges to a limit $l \le \g \d_0$, and $\d_j\to 0$ as $j\to \infty$.
We conclude from \eqref{e35} that
\[
\frac1{r_j^{N+2}}\iint_{B_j\times(t_0-\frac{r_j^2}{l^{p/2-1}}, t_0+\frac{r_j^2}{l^{p/2-1}})}(|\n u(x,t)|^2-l)^{p/2+1}_+dx\,dt
\le \g \d_j^2\to 0\quad (j\to \infty).
\]
Choosing $(x_0,t_0)$ as a Lebesgue point of the function $(|\n u|^2-l)^{p/2+1}_+$ we conclude that $|\n u(x_0,t_0)|\le l^{1/2}\le \g \d_0^{1/2}$
with $\d_0$ estimated in \eqref{d0}.

\renewcommand\thesection{\Alph{section}}
\setcounter{section}{0}

\section{Appendix: Example}

Here we construct a function $f\in L^1(\R^N)$
with compact support such that
$\sup\limits_xW^f_p(x,\infty)<\infty$ however
$\lim\limits_{R\to0}\sup\limits_xW^f_p(x,R)>0$. It is a
generalization of an  example in the celebrated paper by Aizenman and
Simon~\cite[Appendix~1, Example~1]{AS}.

\begin{example}
Let $p\in[2,N)$.
Fix a sequence $\{\rho_n\}\subset (0,1)$ such that
$\rho_n\downarrow  0$ as $n\to\infty$ and
$\sum\limits_n\rho_n^{\frac{N-p}{N-1}}<\infty$, and a bounded
sequence $\{x_n\}\subset \R^N$ such that $|x_n-x_m|\ge
4\rho_{n\wedge m}^{\frac{N-p}{N-1}}$ for $m\ne n$. Let
$f_n:=\rho_n^{-p}\ind_{B_{\rho_n}(x_n)}$ and
$f=\sum\limits_nf_n$. Let $\o_N$ denote the volume of the unit
ball in $\R^N$.
First, note that
\begin{equation}
\label{i}
W^{f_n}_p(x_n,\rho_n)=\int\limits_0^{\rho_n}\frac{dr}r\Big(r^p\rho_n^{-p}\o_N\Big)^{\frac1{p-1}}
=\tfrac{p-1}p\o_N^{\frac1{p-1}}=:a_p.
\end{equation}

Next, let $|x-x_n|< \rho_n +\rho_n^{\frac{N-p}{N-1}}$. Then
\begin{equation}
\label{ii}
W^{f_n}_p(x,\infty)\le W^{f_n}_p(x_n,\infty)= W^{f_n}_p(x_n,\rho_n) +
\int\limits_{\rho_n}^\infty\frac{dr}r\Big(r^{p-N}\rho_n^{N-p}\o_N\Big)^{\frac1{p-1}}
= a_p + \tfrac{p-1}{N-p}\o_N^{\frac1{p-1}}=:b_p.
\end{equation}
Now let $|x-x_n|\ge \rho_n +
      \rho_n^{\frac{N-p}{N-1}}$. Then
      \begin{equation}
      \label{iii}
      \begin{split}
W^{f_n}_p(x,\infty)= \int\limits_{|x-x_n|-\rho_n}^\infty\frac{dr}r\Big(r^{p-N}\rho_n^{-p}
\big|B_r(x)\cap B_{\rho_n}(x_n)\big|\Big)^{\frac1{p-1}}\\
\le \o_N^{\frac1{p-1}}\int\limits_{\rho_n^{\frac{N-p}{N-1}}}^\infty\frac{dr}r\Big(\frac r{\rho_n}\Big)^{\frac{p-N}{p-1}}
= \tfrac{p-1}{N-p}\o_N^{\frac1{p-1}}\rho_n^{\frac{N-p}{N-1}}=: c_p\rho_n^{\frac{N-p}{N-1}}.
\end{split}
      \end{equation}
Observe that if $|x-x_n|<\rho_n +
\rho_n^{\frac{N-p}{N-1}}$ for some $n\in\N$ then
$|x-x_m|>\rho_m + \rho_m^{\frac{N-p}{N-1}}$ for every $m\ne n$. 
Since $p\ge2$,
it follows from \eqref{ii} and \eqref{iii} that
\[
W^{f}_p(x,\infty) \le \sum\limits_nW^{f_n}_p(x,\infty)\le b_p + c_p\sum\limits_n\rho_n^{\frac{N-p}{N-1}}<\infty.
\]
On the other hand, \eqref{i} implies that 
\[
\lim\limits_{R\to0}\sup\limits_{x\in\R^N}W^{f}_p(x,R)\ge \lim\limits_{n\to\infty}W^{f_n}_p(x_n,\rho_n)=a_p>0.
\]
\end{example}

\begin{small}
\paragraph{Acknowledgments.}
The authors
would like to thank Giuseppe Mingione for inspiring discussions with the first named author in Pavia in June 2009,
which made us interested in the topic of this research, and also for the discussion of the content of~\cite{DMl}.
\end{small}

%
\begin{small}

\end{small}


\end{document}